\documentclass[11pt]{article}                   

\usepackage{graphicx}
\usepackage{caption}
\usepackage{subcaption}
\usepackage{xcolor}

\definecolor{blue}{rgb}{0,0,0.7}
\definecolor{red}{rgb}{0.75, 0, 0}
\definecolor{midnight}{rgb}{0.0,0.2,0.4}

\usepackage{hyperref}
\hypersetup{colorlinks=true, citecolor=blue,linkcolor=midnight, urlcolor  = midnight}
\usepackage[all]{xy}           
\usepackage{marginnote}  
\usepackage{float}
\usepackage{tikz-cd}

\usepackage{amsmath,amssymb,amsfonts,amscd,graphicx,psfrag,epsfig,mathabx,amsthm}
\usepackage[titletoc,toc]{appendix}
\usepackage{tikz}
\usepackage{tkz-euclide}
\usetikzlibrary[topaths]
\usetikzlibrary{arrows}
\newcount\mycount
\usepackage{dsfont}
\usepackage{indentfirst}
\usetikzlibrary{matrix}
\numberwithin{equation}{section}
\usepackage[paper=a4paper,left=30mm,right=20mm,top=25mm,bottom=30mm]{geometry}
    \newenvironment{dedication}
        {\vspace{0ex}\begin{quotation}\begin{center}\begin{em}}
        {\par\end{em}\end{center}\end{quotation}}

\newtheorem{theorem}{Theorem}[section]

\newtheorem{remark}[theorem]{Remark}
\newtheorem{theorem-definition}[theorem]{Theorem-Definition}
\newtheorem{theorem-construction}[theorem]{Theorem-Construction}
\newtheorem{lemma-definition}[theorem]{Lemma--Definition}
\newtheorem{lemma-construction}[theorem]{Lemma--Construction}
\newtheorem{lemma}[theorem]{Lemma}
\newtheorem{proposition}[theorem]{Proposition}
\newtheorem{corollary}[theorem]{Corollary}

\newtheorem{definition}[theorem]{Definition}
\newtheorem{example}[theorem]{Example}

\newcommand{\be}{\begin{equation}}
\newcommand{\ee}{\end{equation}}
\newcommand{\lra}{\longrightarrow}

\def\C{\mathbb{C}}
\def\R{\mathbb{R}}
\title{Rational Elliptic Surfaces and the Trigonometry of Tetrahedra}
\author{Daniil Rudenko}
\date{}

\begin{document}
\maketitle 
\begin{dedication}
\hspace{2cm}
\vspace*{0.5cm}{To Sonya Pashchevskaya, the bravest person I know}
\end{dedication}

\tableofcontents

%

\begin{abstract}
We study the trigonometry of non-Euclidean tetrahedra using tools from algebraic geometry. We establish a bijection between non-Euclidean tetrahedra and certain rational elliptic surfaces. We interpret the edge lengths and the dihedral angles of a tetrahedron as values of period maps for the corresponding surface. As a corollary we show that the cross-ratio of the exponents of the solid angles of a tetrahedron is equal to the cross-ratio of the exponents of the perimeters of its faces. The Regge symmetries of a tetrahedron are related to the  action of the Weyl group $W(D_6)$ on the Picard lattice of the corresponding surface.
\end{abstract}

\pagebreak 

\section{Introduction} 

\subsection{Trigonometry of tetrahedra and rational elliptic surfaces}\label{S1.1} 

Trigonometry is a branch of mathematics that studies the relations involving side  lengths and angles of a triangle. It seems that these relations are fairly well understood though some questions remain unanswered, see \cite[pp. 189-194]{Kle16} and \cite[\S 2]{Tju75}.   The situation in higher dimensions is much more complicated. 

By a {\it tetrahedron} we mean a geodesic tetrahedron in $\mathbb{S}^3$, $\mathbb{R}^3$	or $\mathbb{H}^3.$ We call a tetrahedron {\it non-Euclidean} if it is spherical or hyperbolic. The following problem is a subject of three-dimensional trigonometry:\\

\begin{center} \begin{minipage}{30em} 
	How can one determine the dihedral angles of a tetrahedron from its edge lengths?\\
\end{minipage}
\end{center}

This problem admits a straightforward solution: one can write a complicated explicit formula, presenting the dihedral angles as functions of the lengths of edges. Surprisingly, this is not the end of the story. We start with formulating a theorem in three-dimensional trigonometry, which motivated this work.

Let $T$ be a tetrahedron with vertices $A_1, A_2, A_3, A_4.$ Denote by $l_{ij}\in \R$ for $1\leq i< j\leq 4$ the length of an edge $A_i A_j$ and by  $\alpha_{ij}\in \mathbb{R}/2\pi\mathbb{Z}$ the corresponding dihedral angle. Next, consider solid angles $\Omega_{123}, \Omega_{124},\Omega_{134}, \Omega_{234}$ and perimeters  $\Pi_{123}, \Pi_{124}, \Pi_{134}, \Pi_{234}$ of its faces. More explicitly,
\begin{align*}
&\Pi_{123}=l_{12}+l_{13}+l_{23}, \ \ \ \ \Omega_{123}=\alpha_{14}+\alpha_{24}+\alpha_{34}-\pi,\\
&\Pi_{124}=l_{12}+l_{14}+l_{24}, \ \ \ \  \Omega_{124}=\alpha_{13}+\alpha_{23}+\alpha_{34}-\pi, \\
&\Pi_{134}=l_{13}+l_{14}+l_{34}, \ \ \ \ \Omega_{134}=\alpha_{12}+\alpha_{23}+\alpha_{24}-\pi,\\
&\Pi_{234}=l_{23}+l_{24}+l_{34}, \ \ \ \ \Omega_{234}=\alpha_{12}+\alpha_{13}+\alpha_{14}-\pi.
\end{align*}
We also consider the following quantities: 
\begin{align*}
&\Pi_{1234}=l_{12}+l_{23}+l_{34}+l_{41}, \ \ \ \ \Omega_{1234}=\alpha_{12}+\alpha_{23}+\alpha_{34}+\alpha_{41},\\
&\Pi_{1324}=l_{13}+l_{32}+l_{24}+l_{41},\ \ \ \ \Omega_{1324}=\alpha_{13}+\alpha_{32}+\alpha_{24}+\alpha_{41}, \\
&\Pi_{1243}=l_{12}+l_{24}+l_{43}+l_{31}, \ \ \ \ \Omega_{1243}=\alpha_{12}+\alpha_{24}+\alpha_{43}+\alpha_{31}.
\end{align*}
We assemble these numbers into a pair of configurations of eight points in $\mathbb{P}^1=\C\cup\{\infty\}.$ First, we consider a configuration
\begin{align*}
\Omega(T)=\left(1,e^{i\Omega_{123}}, e^{i\Omega_{124}}, e^{i\Omega_{134}}, e^{i\Omega_{234}}, e^{i\Omega_{1234}}, e^{i\Omega_{1324}}, e^{i\Omega_{1243}}\right). 
\end{align*}
Next, we consider a configuration
\begin{align*}
\Pi(T)=
\begin{cases}
\Bigl(1,e^{i\Pi_{123}},e^{i\Pi_{124}},e^{i\Pi_{134}},e^{i\Pi_{234}},e^{i\Pi_{1234}},e^{i\Pi_{1324}}, e^{i\Pi_{1243}}\Bigr)& \text{ if } T \text{ is spherical},	\\
\left(0,\: \Pi_{123},\: \Pi_{124},\: \Pi_{134},\: \Pi_{234},\: \Pi_{1234},\: \Pi_{1324},\: \Pi_{1243}\right) & \text{ if } T \text{ is Euclidean},	\\
\bigl(1,\; e^{\Pi_{123}},\;  e^{\Pi_{124}},\; e^{\Pi_{134}},\; e^{\Pi_{234}},\; e^{\Pi_{1234}},\; e^{\Pi_{1324}},\;  e^{\Pi_{1243}}\: \bigr)  & \text{ if } T \text{ is hyperbolic}.	
\end{cases}
\end{align*}

\begin{theorem} \label{TheoremProjectiveEquivalence} For a tetrahedron $T$ the configurations $\Omega(T)$ and $\Pi(T)$ are projectively equivalent.
\end{theorem}

For distinct $z_1, z_2, z_3, z_4\in \C$ consider a cross-ratio
$[z_1, z_2, z_3, z_4]=\dfrac{(z_1-z_2)(z_3-z_4)}{(z_1-z_4)(z_3-z_2)} \in \C^{\times}.$ Projective transformations preserve cross-ratios,  so the following corollary holds.

\begin{corollary}  \label{TheoremCrossRatioIdentity}
For a tetrahedron $T$ the following equality holds:
\[
\left[e^{i\Omega_{123}}, e^{
i\Omega_{124}}, e^{i\Omega_{134}}, e^{i\Omega_{234}}\right]=
\begin{cases}
	\left[e^{i\Pi_{123}}, e^{i\Pi_{124}}, e^{i\Pi_{134}}, e^{i\Pi_{234}}\right] &\text{if $T$ is spherical,}\\
	\left[\Pi_{123},\: \Pi_{124},\: \Pi_{134},\: \Pi_{234}\right] &\text{if $T$ is Euclidean,}\\
	\left[e^{\Pi_{123}}, e^{\Pi_{124}}, e^{\Pi_{134}},e^{\Pi_{234}} \right] &\text{if $T$ is hyperbolic.}
\end{cases}
\]
\end{corollary}

We discovered Theorem \ref{TheoremProjectiveEquivalence} in an attempt to understand a formula of Cho and Kim for volume of a non-Euclidean tetrahedron (see \cite{ChK99}) from a motivic perspective. For an elementary proof of Corollary \ref{TheoremCrossRatioIdentity} see an answer by Petrov in the discussion {\it``A curious relation between angles and lengths of edges of a tetrahedron''} on \url{https://mathoverflow.net/q/336464}. It is of interest to understand the geometric meaning of the coefficients of the projective transformation sending  $\Omega(T)$  to $\Pi(T).$

The correspondence between the dihedral angles and the lengths of edges of a tetrahedron has a hidden symmetry called {\it Regge symmetry}, discovered by Ponzano and Regge in the Euclidean case, see \cite[Appendix D]{PR68}.

\begin{theorem} \label{TheoremReggeSymmetry} Let $T$ be a tetrahedron.
 Suppose that there exists a tetrahedron $T'$ in the same space with edge lengths $l_{ij}'$  for $1\leq i<j\leq4$ such that
\begin{align*}
&l_{12}'=l_{12}, \ \  \
 l_{13}'=\frac{l_{14}+l_{23}+l_{24}-l_{13}}{2}, \ \ \
 l_{14}'=\frac{l_{13}+l_{23}+l_{24}-l_{14}}{2},\\
&l_{34}'=l_{34}, \ \ \
 l_{24}'=\frac{l_{13}+l_{14}+l_{23}-l_{24}}{2}, \ \ \
 l_{23}'=\frac{l_{13}+l_{14}+l_{24}-l_{23}}{2}.
\end{align*}
Then the  corresponding dihedral angles $\alpha_{ij}'$ of $T'$ satisfy
\begin{align*}
& \alpha_{12}'=\alpha_{12},\ \  
  \alpha_{13}'=\frac{\alpha_{14}+\alpha_{23}+\alpha_{24}-\alpha_{13}}{2},\ \
  \alpha_{14}'=\frac{\alpha_{13}+\alpha_{23}+\alpha_{24}-\alpha_{14}}{2},\\
& \alpha_{34}'=\alpha_{34},\ \
  \alpha_{24}'=\frac{\alpha_{13}+\alpha_{14}+\alpha_{23}-\alpha_{24}}{2},\ \
  \alpha_{23}'=\frac{\alpha_{13}+\alpha_{14}+\alpha_{24}-\alpha_{23}}{2}.
\end{align*}
Moreover, the volumes of the tetrahedra $T$ and $T'$ coincide.
\end{theorem}

A geometric proof of Theorem \ref{TheoremReggeSymmetry} in the non-Euclidean case was found by Akopyan and Izmestiev, see \cite[Theorem 1]{AI19}. It was noticed in \cite{DL03} that Regge symmetry is a part of a bigger group of order $23040,$ which is isomorphic to the Weyl group $W(D_6).$ 

Our initial goal was to find a conceptual explanation for Theorems \ref{TheoremProjectiveEquivalence} and \ref{TheoremReggeSymmetry}. Our main result is a construction of a correspondence between tetrahedra and certain complex projective surfaces. The cross-ratios from Theorem \ref{TheoremCrossRatioIdentity} are equal to the classical invariants of the surfaces, called cross-ratios of type $D_4$ in \cite[\S 3]{Nar80}. Regge symmetry is manifested in the Weyl group action on the Picard lattice of the surface. 

By a rational elliptic surface $X$ we mean a smooth projective surface over $\mathbb{C},$ which can be obtained as a blow up of $\mathbb{P}^2$ at nine points of intersection of a pair of elliptic curves. The anti-canonical linear system $|-K_X|$ defines a map $X\lra \mathbb{P}^1$ with generic fiber of genus $1.$ A fiber $F$ of the elliptic fibration on $X$ is said to have type $I_2$ if it is isomorphic to a union of two rational curves intersecting transversally at a pair of points. The group $\textup{Pic}^0(F)$ can be identified  with $\mathbb{C}^{\times}.$ Consider a divisor $D,$ which is orthogonal (with respect to the intersection pairing on $\textup{Pic}(X))$ to each irreducible component of $F.$ We denote by 
\[
\textup{Res}_F(D)\in \textup{Pic}^0(F)\cong \mathbb{C}^{\times}
\]
the restriction of $D$ to $F.$ The map $\textup{Res}_F$ gives a natural way to parametrize rational elliptic surfaces, see \cite[Appendix by E. Looijenga]{Nar82}, we call it a {\it period map}.

\begin{theorem}\label{TheoremHyperbolic}
For a generic non-Euclidean tetrahedron  (see Definition \ref{DefinitionGeneric}) there exists a rational elliptic surface $X_T$ with a pair of $I_2-$fibers $F_1$ and $F_2$ and a collection of six classes $e_{ij}\in \textup{Pic}\left(X_T\right)$  orthogonal to $F_1$ and $F_2$ such that for $1\leq i <j\leq 4$ we have
\be \label{FormulaResidue}
\begin{split}
	&\textup{Res}_{F_1}(e_{ij})=
\begin{cases}
	e^{2il_{ij}}& \text{if $T$ is spherical},\\
	e^{2l_{ij}}&\text{if $T$ is hyperbolic},
\end{cases}\\
&\textup{Res}_{F_2}(e_{ij})=e^{2 i(\pi- \alpha_{ij})}.
\end{split}
\ee
\end{theorem}

\begin{remark}One can show that Theorem \ref{TheoremHyperbolic} is true without an extra assumption that $T$ is generic. Moreover, we expect that Theorem \ref{TheoremHyperbolic} can be generalized to the case of a Euclidean tetrahedron. The corresponding surface $X_T$  has a fiber $F_1$ of type $III$ and a fiber $F_2$ of type $I_2.$ Then $\textup{Pic}^0(F_1)\cong\C$ and we have
\begin{align*}
	&\textup{Res}_{F_1}(e_{ij})=l_{ij},\\
	&\textup{Res}_{F_2}(e_{ij})=e^{2 i(\pi-\alpha_{ij})}.
\end{align*}
\end{remark}

First we explain that Theorem \ref{TheoremReggeSymmetry} follows from Theorem \ref{TheoremHyperbolic}. The lattice $\textup{Pic}\left(X_T\right)$ has rank $10;$ the orthogonal complement to the canonical class $K_X \in \textup{Pic}\left(X_T\right)$ is an affine root lattice of type $E_8^1.$ The orthogonal complement to all components of the fibers $F_1$ and $F_2$ contains the null-vector $K_{X_T}$ and its quotient by $K_{X_T}$ is a root lattice of type $D_6,$ so the Weyl group $W(D_6)$ acts on it. Regge symmetry is an action on this quotient by a particular element of the Weyl group, namely the reflection with respect to the plane perpendicular to the root $\dfrac{e_{13}+e_{14}+e_{23}+e_{24}}{2}$. Since the period maps (\ref{FormulaResidue}) in Theorem \ref{TheoremReggeSymmetry} are linear, Regge symmetry transforms lengths of edges and dihedral angles according to the formulas in Theorem \ref{TheoremReggeSymmetry}.

Theorem \ref{TheoremProjectiveEquivalence} also follows from Theorem \ref{TheoremHyperbolic}. A rational elliptic surface  $X_T$ carries an admissible conic bundle: a map  $b\colon X_T \lra \mathbb{P}^1$ with  generic fiber of genus $0$ such that $b$ sends each irreducible component of  $F_1$ and $F_2$ isomorphically to $\textup{Im}(b)\cong\mathbb{P}^1$.  Map $b$ has eight critical values $p_1,\dots, p_8\in \textup{Im}(b).$ We choose $b$ such that $b^{-1}(p_i)$ intersects each component of $F_1$ and $F_2$ in exactly one point. For a certain choice of a conic bundle the eight points in which $b^{-1}(p_i)$ intersect a component of $F_1$ form a configuration $\Pi(T)$ and the eight points in which $b^{-1}(p_i)$ intersect a component of $F_2$ form a configuration $\Omega(T).$ The map $b$ defines a projective transformation sending $\Pi(T)$ to $\Omega(T).$ This proves Theorem \ref{TheoremProjectiveEquivalence}.

\subsection{Projective tetrahedra and the $E_8$ lattice} \label{SectionIntroductionTetrahedra}
We start with introducing an algebro-geometric avatar of a non-Euclidean tetrahedron. A projective tetrahedron $T=(Q,\mathcal{H})$ is a configuration of an irreducible quadric $Q$ and an ordered set of four planes $\mathcal{H}=\{H_1, H_2, H_3, H_4\}$ in $\mathbb{P}^3$ satisfying a certain non-degeneracy condition, see  Definition \ref{DefinitionProjectiveTetrahedra}. 

Every non-Euclidean tetrahedron defines a projective tetrahedron, see \cite[\S 1.5]{Gon99}. We describe here the hyperbolic case. In  Klein's  model the hyperbolic space $\mathbb{H}^3$ is identified with the interior of the unit ball in $\mathbb{R}^3$ and geodesic subspaces of $\mathbb{H}^3$ are intersections of lines and planes in $\mathbb{R}^3$ with $\mathbb{H}^3.$ We  view $\mathbb{R}^3$ as the set of real points of an affine space inside the complex projective space $\mathbb{P}^3.$ Let $Q$ be a quadric, obtained as a projectivization  of the complexification of the ideal boundary $\partial \mathbb{H}^3.$ Let $H_1, H_2, H_3, H_4$ be the projectivizations  of the complexifications of the faces of the tetrahedron. In this way a geodesic tetrahedron in $\mathbb{H}^3$ determines the projective tetrahedron $(Q,\{H_1,H_2,H_3,H_4\}).$ A {\it marking} of a projective tetrahedron is a choice of a family of lines on $Q\cong \mathbb{P}^1\times \mathbb{P}^1$ and ordering of points $\{x, y\}=Q\cap (A_iA_j).$ We denote the first point in $Q\cap (A_iA_j)$ by $E_{ij}$ and the second by $E_{ji}.$ 

A projective tetrahedron obtained from a non-Euclidean tetrahedron admits a canonical marking. In the hyperbolic case we have
\[
\log \left ( [A_i, x, A_j, y] \right )=\pm 2l_{ij}
\]
and we choose the ordering of $\{x,y\}$ such that $\log \left ( [x, A_i, y, A_j] \right )>0.$
In the spherical case
\[
\log \left ( [A_i, x, A_j,y] \right )=\pm 2 i \alpha_{ij}
\]
and we choose the  ordering of $\{x,y\}$ such that $\frac{1}{2}\log \left ( [x, A_i, y, A_j] \right )\in (0,\pi).$

A non-Euclidean tetrahedron is uniquely determined by its edge lengths. One might expect that a marked projective tetrahedron is uniquely determined by the quantities 
\begin{align} \label{FormulaEdgeLength}
	[A_i, E_{ij}, A_j, E_{ji}]\in \C^{\times}
\end{align}
for $1\leq i<j\leq 4,$ but this is not the case. The reason is that quantities like $e^{\Pi_{123}}$ appearing in $\Pi(T)$ are well defined, but are not rational functions of expressions like (\ref{FormulaEdgeLength}).

Here the lattice $\textup{Q}\bigl(\textup{E}_8\bigr)$ of a root system of type $E_8$  comes into play. It is known that it contains a set of eight pairwise orthogonal roots, which we label by subsets of a set $I=\{1,2,3,4\}$ of even cardinality. The stabilizer of a set of eight orthogonal roots acts $3-$transitively on them. After fixing  roots $e_{\varnothing}$ and $e_{I}$ the remaining six could be labelled by $2-$subsets of the set $\{1,2,3,4\}.$ The stabilizer of the set of eight roots, $e_{\varnothing}$ and  $e_{I}$ acts as $\mathbb{S}_4$ on this set. We call the remaining roots  $e_{12}, e_{13}, e_{14}, e_{23}, e_{24}, e_{34}.$ 

The collection of roots orthogonal to $e_{I}$ spans a lattice of type $E_7,$ which we denote by $\textup{Q}\bigl(\textup{E}_7^\textup{L}\bigr)$. In \S \ref{SectionEdgeLengths} we construct a homomorphism
\[
\textup{L}_T  \colon \textup{Q}\bigl(\textup{E}_7^\textup{L}\bigr) \lra \C^{\times},
\]
such that $\textup{L}_T(e_{\varnothing})=1$ and 
\[
\textup{L}_T(e_{ij})=[ A_i, E_{ij}, A_j, E_{ji}]
\] 
for $1\leq i<j\leq 4.$ If $T$ is obtained from a non-Euclidean tetrahedron, the coordinates of the vector $\Pi(T)$  are values of the map $\textup{L}_T$ on certain roots in $\textup{E}_7^\textup{L}.$ We call $T$ {\it generic} if the only roots $r\in \textup{R}\bigl(\textup{E}_7^\textup{L}\bigr)$ for which  $\textup{L}_T(r)=1$ equal to $\pm e_\varnothing.$

Similarly, denote by $\textup{Q}\bigl(\textup{E}_7^\textup{A}\bigr)$ an orthogonal complement to the root $e_I$ in $\textup{Q}\bigl(\textup{E}_8\bigr)$. In \S  \ref{SectionEdgeLengths} we define an angle function
\[
\textup{A}_T \colon \textup{Q}\bigl(\textup{E}_7^\textup{A}\bigr) \lra \C^{\times}.
\]
This function is closely related to the length function of the dual tetrahedron $T^{\vee}.$  If  $T$ is obtained from a non-Euclidean tetrahedron we have $\textup{A}_T(e_{\varnothing})=1$ and
\[
\textup{A}_T(e_{ij})=e^{2i(\pi-\alpha_{ij})} \text{ for } 1\leq i<j\leq 4.
\] 
Similarly to $\Omega(T)$ and $\Pi(T)$ we define configurations 
\begin{align*}
&\textup{A}(T)=\left(1,\textup{A}_{123}, \textup{A}_{124}, \textup{A}_{134}, \textup{A}_{234}, \textup{A}_{1234}, \textup{A}_{1324}, \textup{A}_{1423}\right),\\
&\\
&\textup{L}(T)=\left(1,\textup{L}_{123}, \textup{L}_{124}, \textup{L}_{134}, \textup{L}_{234}, \textup{L}_{1234}, \textup{L}_{1324}, \textup{L}_{1423}\right). 
\end{align*}
of eight points in $\mathbb{P}^1=\C\cup\{\infty\},$ see \S \ref{SectionEdgeLengths}. 

\begin{definition}\label{DefinitionChoKim} Let $T$ be a marked projective tetrahedron. The following two rational functions
\begin{align*}
&\textup{CK}^{\textup{L}}_T(t)=\frac{(t-\textup{L}_{123})(t-\textup{L}_{124})(t-\textup{L}_{134})(t-\textup{L}_{234})}{(t-1)(t-\textup{L}_{1234})(t-\textup{L}_{1324})(t-\textup{L}_{1243})},\\ \\
&\textup{CK}^{\textup{A}}_{T}(t)=\frac{(t-\textup{A}_{123})(t-\textup{A}_{124})(t-\textup{A}_{134})(t-\textup{A}_{234})}{(t-1)(t-\textup{A}_{1234})(t-\textup{A}_{1324})(t-\textup{A}_{1243})}\\
\end{align*}
are called the  {\it Cho-Kim function of $T$} and {\it the dual Cho-Kim function of $T$} respectively. 
\end{definition}

It is easy to see that $\textup{CK}^{\textup{L}}_T(0)=\textup{CK}^{\textup{L}}_T(\infty)=1.$
There exist two more points $p_1, p_2 \in \mathbb{P}^1$ such that $\textup{CK}^{\textup{L}}_T(p_1)=\textup{CK}^{\textup{L}}_T(p_2)=1.$  These numbers are called the {\it principal parameters} of $T.$  The Cho-Kim function first appeared in a surprising formula for the volume of a hyperbolic tetrahedron in \cite{MM05}.

Our main result about trigonometry of a projective tetrahedron is the following theorem, which immediately implies Theorem \ref{TheoremProjectiveEquivalence}.

\begin{theorem} \label{TheoremProjectiveTetrahedra}
There exists a  unique fractional linear transformation $\psi \in \textup{PSL}_2(\C)$ such that 
\[
\textup{CK}^{\textup{L}}_{T}(t)=\textup{CK}^{\textup{A}}_{T}\left(\psi(t)\right).
\]
For a certain order of the principal parameters $p_1$ and $p_2$ we have
\[
\psi(t)=\frac{(t-p_1)(1-p_2)}{(t-p_2)(1-p_1)}.
\]
In particular, the configurations $\textup{L}(T)$ and $\textup{A}(T)$ are projectively equivalent.
\end{theorem}

Theorem \ref{TheoremProjectiveTetrahedra} allows one to {\it solve} a tetrahedron: compute its dihedral angles in terms of edge lengths. 

\begin{example}\label{ExampleSphericalTetrahedron}
	Consider a spherical  tetrahedron with all dihedral angles equal to $\dfrac{\pi}{2}.$ Let $T$ be the corresponding projective tetrahedron. The dual Cho-Kim function of $T$ is equal to
\[
\textup{CK}_T^{\textup{A}}(t)=\frac{\left(t-i\right)^4}{(t-1)^4}
\]
The principal parameters of $T$ are $p_1=1+i$ and $p_2=\dfrac{1+i}{2},$ so 
$
\psi(t)=\frac{t(1-i)-1}{t-(i+1)}.
$
By Theorem \ref{TheoremProjectiveTetrahedra} we have
\[
\textup{CK}_{T}^{\textup{L}}(t)=\frac{(t+i)^4}{(t-1)^4}
\]
and the edge lengths of $T$ equal to $\dfrac{3\pi}{2}.$	
\end{example}

\subsection{Projective tetrahedra and $D_6$-surfaces}\label{S1.4} 

We call a rational elliptic surface $X$ with a pair of $I_2-$fibers $F_1$ and $F_2$ a {\it $D_6$-surface} $(X,F_1,F_2)$. A $D_6$-surface is called {\it generic } if all other fibers are irreducible. Fix an ordering of the components and of the singular points of each fiber $F_1$ and $F_2.$  After that we can identify groups $\textup{Pic}^0(F_1)$ and $\textup{Pic}^0(F_2)$  with $\C^{\times}.$ Lattice $\textup{Pic}(X)$ is isomorphic to $\mathbb{Z}^{10},$ moreover $K_X^2=0.$ It is known that the lattice $K_X^{\perp}/K_X$ is a root lattice of type $E_8.$ Fix an isomorphism $K_X^{\perp}/K_X \cong \textup{Q}(\textup{E}_8)$ sending  the classes of the chosen components of the fibers $F_1$ and $F_2$ to the roots $e_I$ and $e_{\varnothing}$ respectively. We call a choice of this isomorphism together with an ordering of the components and of the singular points of $F_1$ and $F_2$ {\it a marking} of $X.$ The Weyl group $W(\textup{D}_6)$ of order $23040$ is a stabilizer of roots $e_I$ and $e_{\varnothing}$ in $W(\textup{E}_8)$ and thus acts on the set of markings of $X$ by changing the isomorphism $K_X^{\perp}/K_X \cong \textup{Q}\left(\textup{E}_8\right).$ For a marked $D_6$-surface the following period maps are defined:
\begin{align*}
&\textup{Res}_{F_1} \colon  \textup{Q}\bigl(\textup{E}_7^\textup{L}\bigr) \lra \textup{Pic}^{0}(F_1),\\
&\textup{Res}_{F_2} \colon \textup{Q}\bigl(\textup{E}_7^\textup{A}\bigr)   \lra \textup{Pic}^{0}(F_2).
\end{align*}

The main result of our paper is Theorem \ref{TheoremCorrespondence}, the generalization of Theorem  \ref{TheoremHyperbolic} below.

\begin{theorem}\label{TheoremCorrespondence}
	There exists a $W(D_6)$-equivariant one-to-one correspondence $T \longleftrightarrow (X_T,F_1,F_2)$ between generic marked projective tetrahedra 
		and generic marked	$D_6$-surfaces such that $\textup{L}_T=\textup{Res}_{F_1}$ and $\textup{A}_T=\textup{Res}_{F_2}$.
\end{theorem}

The construction of the surface $X_T$ from the tetrahedron $T$ consists of two steps. First we blow up the quadric $Q$ at the twelve points \[\bigcup_{1\leq i<j\leq 4} \bigl (H_i \cap H_j \cap Q\bigr)\] and obtain a rational surface $R_T.$ Next, we blow down a certain set of four non-intersecting $(-1)-$curves and obtain a rational elliptic surface $X_T$. The last step is not canonical.  Surprisingly,  different choices of four $(-1)-$curves result in isomorphic rational elliptic surfaces. Unfortunately, we do not have a clear explanation of this fact yet; its proof is based on the Torelli theorem for anti-canonical pairs. The details of the construction are presented in \S~\ref{SectionConstruction}. In the sequel to this paper we will give a  ``motivic'' proof of Theorem~\ref{TheoremCorrespondence} based on an isomorphism of mixed Hodge structures
\begin{align*}
		H^3 \Bigl ( \mathbb{P}^3 \smallsetminus Q, \bigcup_{i=1}^4 H_i\Bigr)\cong
	H^2(X_T\smallsetminus F_1, F_2).
\end{align*}

Theorem~\ref{TheoremProjectiveTetrahedra} is a corollary of Theorem~\ref{TheoremCorrespondence}. For the generic $D_6$-surface $X_T$ there exists  a map
\[
b\colon X_T \lra \mathbb{P}^1
\]
such that the restriction of $b$ to each of the four irreducible components of the fibers $F_1$ and $F_2$ is an isomorphism. For almost every point $p\in \mathbb{P}^1$ the fiber $b^{-1}(p)$ is a rational curve with four marked points of intersection with the components of the fibers $F_1$ and $F_2.$ The cross-ratio of the points is a rational function on the target space $\mathbb{P}^1$ of $b.$ To write it down explicitly we need to fix three points $0, 1$ and $\infty$ on the target. We will see that for one such choice this rational function is equal to the Cho-Kim function $\textup{CK}^{\textup{L}}_T$, and for another, the dual  Cho-Kim function $\textup{CK}^{\textup{A}}_T.$ This immediately implies Theorem \ref{TheoremProjectiveTetrahedra}.

The structure of possible configurations of degenerate fibers of a rational elliptic surface is well understood, see \cite{KS91}, \cite{Per90}. It is interesting to extend Theorem \ref{TheoremCorrespondence} to non-generic $D_6-$surfaces and  non-generic non-Euclidean tetrahedra:  tetrahedra with ideal vertices, Euclidean tetrahedra, disphenoids etc. 

\begin{example}
Consider the spherical  tetrahedron from Example \ref{ExampleSphericalTetrahedron}.
Both $\textup{L}_T$ and $\textup{A}_T$ take values $\pm 1$ on the roots. The corresponding rational elliptic surface has four singular fibers: two of type $I_2$ and two of type $I_4.$ There is a unique surface $X_{4422}$ with these types of fibers and it has eight $(-1)-$curves, see \cite[\S5]{MP86}. 
\end{example}

\subsection{Notation and conventions} 

Throughout this paper we work over $\mathbb{C}$.  For distinct points $P_1, P_2\in \mathbb{P}^3$ we denote by $(P_1P_2)$ the line  containing them. Similarly, for three points  $P_1, P_2, P_3$ in general position we denote by $(P_1P_2P_3)$ the plane containing them. For a point $P$ and a line $l$ not containing $P$ we denote by $\langle P, l \rangle$ the plane spanned by $P$ and $l$. Finally, for a pair of intersecting distinct lines $l_1, l_2$ we denote by $\langle l_1, l_2 \rangle$ the plane that contains them. 

The cross-ratio of four points $P_1,P_2, P_3, P_4$ on a rational curve $C$ is denoted  $[P_1, P_2, P_3, P_4]_C.$ Consider a birational isomorphism  $\varphi \colon X \dashrightarrow Y$  of smooth projective surfaces $X$ and $Y.$ For a smooth projective curve $C$  in $X$ we define
\[
\varphi(C)=\overline{(C \smallsetminus \textup{Ind}(\varphi))},
\]  
where $\textup{Ind}(\varphi)$ is the locus of indeterminacy of $\varphi.$ Then $\varphi(C)$ is a point or
the map $\varphi$ defines an isomorphism $C \lra \varphi(C),$ which we denote by the same letter. If $C$ is rational and $\varphi(C)$ is not a point, then for every four points $P_1, P_2, P_3, P_4 \in C$ we have
\[
[P_1, P_2, P_3, P_4]_C=[\varphi(P_1), \varphi(P_2), \varphi(P_3), \varphi(P_4)]_{\varphi(C)}.
\]
If for  $E\in \textup{Pic}(X)$ there exists a unique curve $C$ on $X$ such that $[C]=E,$ then we will denote the curve by the same letter $E.$

Finally, for a root system $\mathcal{R}$ we denote the corresponding lattice by $\textup{Q}(\mathcal{R})$ and the set of roots by $\textup{R}(\mathcal{R}).$ We adopt a convention, according to which we have $r^2=-2$ for $r\in \textup{R}(\mathcal{R}).$ We mainly work with a concrete root system of type  $E_8$ and its subsystems (see \S \ref{SectionE8RootSystem}), which we denote by roman $\textup{E}_8.$

\paragraph{Acknowledgments}  I would like to thank F. Brown, I. Dolgachev, O. Martin,  E. Looijenga, and A.~Goncharov  for motivating discussions and invaluable help with preparing this manuscript. I am very grateful to P.~Deligne who read a preliminary version of this paper and made a lot of useful comments and suggestions. I also
thank Gerhard Paseman for checking some of the details in an earlier draft.

\section{Trigonometry of projective tetrahedra}\label{SecT}
\subsection{Projective tetrahedra} \label{SectionProjectiveTetrahedra}

\begin{definition}\label{DefinitionProjectiveTetrahedra} A projective tetrahedron $T=(Q,\mathcal{H})$ is a configuration, consisting of a smooth quadric $Q$  in $\mathbb{P}^3$ and an ordered set of four planes $\mathcal{H}=\{H_1, H_2, H_3, H_4\}$ in general position such that conics $Q \cap H_i$ are smooth and points $A_i=\bigcap_{j\neq i}H_j$ do not lie on $Q.$ Planes $H_i$ are called {\it faces} of $T,$ lines $H_i \cap H_j$ are called {\it edges} of $T,$ and  points $A_i$  are called {\it vertices} of $T$.
\end{definition}

A smooth quadric in $\mathbb{P}^3$ is isomorphic to $\mathbb{P}^1 \times \mathbb{P}^1$ and defines a duality between subspaces of $\mathbb{P}^3$ known as the {\it polar duality}. Points, lines, and planes are dual to planes, lines, and points respectively. A general line $l$ intersects $Q$ in two points; ordering of these points is called an {\it orientation} of $l$. Lines contained in $Q$ are called {\it generators} of the quadric; there are two families of generators. An ordering of these families is called an {\it orientation} of $Q.$ 

Every point $p\in Q$ is contained in exactly two generators, which belong to the different families. After an orientation of $Q$ is fixed, we denote them $L_p$ and $R_p$ and call {\it the left generator} and {\it the right generator} respectively. Generators of the quadric are self-dual lines. Let $l$ be a line in $\mathbb{P}^3$ which intersects $Q$ transversally at points $x$ and $y$. Then $Q\cap l=\{x,y\}$ and the dual line $l^{\vee}$ is the unique line passing through the points $L_x\cap R_y$ and $R_x\cap L_y.$ 

Let $T=(Q,\mathcal{H})$ be a projective tetrahedron in $\mathbb{P}^3$. The dual tetrahedron $T^{\vee}=(Q,\mathcal{A})$ is given by the configuration consisting of the same quadric $Q$ and planes $\mathcal{A}=\{A_1^{\vee}, A_2^{\vee},A_3^{\vee},A_4^{\vee}\}$ in $\mathbb{P}^3$ dual to the vertices of $T.$  Notice that the edges of $T$ are dual to the edges of $T^{\vee}.$

\begin{definition}
	A {\it marking} of a projective tetrahedron $T=(Q,\mathcal{H})$ is combinatorial data consisting of an orientation of the quadric $Q$ and orientations of the edges of $T.$
\end{definition}
Denote by $E_{ij}$ the first point of $(A_iA_j)\cap Q$ and by $E_{ji}$ the second point. Every marking of a tetrahedron $T$ determines a marking of the dual tetrahedron $T^{\vee}$ in the following way.  Points $(A_iA_j)^{\vee}\cap Q$ are ordered so that $E_{ij}'=L_{E_{ij}}\cap R_{E_{ji}}$ is the first and $E_{ji}'=R_{E_{ij}}\cap L_{E_{ji}}$ is the second. It is easy to see that projective duality is an involution on marked projective tetrahedra.
 

\subsection{The  root system $\textup{E}_8$} \label{SectionE8RootSystem}
A  root system of type $E_8$ consists of $240$ vectors in $\mathbb{R}^8,$ see \cite[\S 6.4.10]{Bou68}. It contains a set of $8$ orthogonal roots and the Weyl group $W(E_8)$ acts transitively on such sets, see \cite[Proposition 2.1]{DM10}.

\begin{proposition} 
In a root system of type $E_8$, let $S$ be a set of $8$ orthogonal roots. Then $S$ has a natural structure of an affine space of dimension $3$ over $\mathbb{F}_2.$ The planes for this structure are sets $P$ of $4$ elements of $S$ such that $\frac{1}{2}\sum_{\alpha \in P}\alpha$ is a root of $E_8$ and the stabilizer of $S$ in $W(E_8)$ is the group of affine transformations.
\end{proposition}
\begin{proof}
See \cite[Theorem 2.5]{DM10}.	
\end{proof}

There is a way to construct an $E_8$ root system out of an affine  space $S$ over $\mathbb{F}_2$ of dimension $3.$ Consider the following subset $C\subset\left (\mathbb{Z}/2\mathbb{Z}  \right)^S:$
\[
C=\left\{0, \sum_{s\in S} e_s, \sum_{s\in P} e_s \text{ for }P \text{ a plane in }S\right\}.
\]
Then $C$ is a subgroup: if $P, P'$ are planes in $S,$ their symmetric difference is the empty set, $S$, or a plane in $S.$ In a lattice $\left(\frac{1}{2}\mathbb{Z}\right)^S$	 with quadratic form $-2\bigl(\sum x_i^2\bigr)$ consider the subspace $\mathbb{Z}^S \subset E \subset \left(\frac{1}{2}\mathbb{Z}\right)^S$ given by 
\[
E=\{x \mid 2x \text{ has an image in } \left(\mathbb{Z}/2\mathbb{Z}\right)^S  \text{ which is in }C \}.
\]
It is not hard to see that $E$ is a lattice of type $E_8.$ The roots are vectors $\pm e_s$ for $s\in S$ and $\dfrac{\pm e_a \pm e_b \pm e_c \pm e_d}{2}$
for $\{a,b,c,d\}$ -- an affine plane.

For a set $I=\{1,2,3,4\}$  consider an affine space $S$ of even subsets of $I$. Explicitly,
\[
S_I=\Big \{\varnothing, \{1,2\}, \{1,3\}, \{1,4\}, \{2,3\}, \{2,4\}, \{3,4\},I\Big\}.
\]
From now on we denote by (roman) $\textup{E}_8$ the root system obtained from the affine space $S_I$ by the construction above.  We put  $e_{ij}:=e_{\{i,j\}}\text{ for } i<j.$  Next, we denote the orthogonal complement to the root $e_I$ in $\textup{E}_8$ by $\textup{E}_7^{\textup{L}}$ and  the orthogonal complement to  $e_{\varnothing}$  by $\textup{E}_7^{\textup{A}}.$ These  are root systems of type $E_7$ and their intersection $\textup{D}_6=\textup{E}_7^{\textup{L}}\cap \textup{E}_7^{\textup{A}}$ is a root  system of type $D_6.$  A map $D \colon 	S_I \lra S_I$ sending a subset of $I$ to its complement is affine and so lies in the Weyl group $W(\textup{E}_8).$ It defines an isometry $D\colon \textup{Q}\bigl(\textup{E}_7^\textup{L}) \lra \textup{Q}\bigl(\textup{E}_7^\textup{A}\bigr )$ which leaves the lattice $\textup{Q}\bigl(\textup{D}_6\bigr )$  invariant.

\subsection{Edge function and angle function} \label{SectionEdgeLengths}
We start with discussing an algebro-geometric counterpart of the Poincare model of hyperbolic geometry. Consider a double cover of the projective space $\mathbb{P}^3$ ramified at $Q.$ This is a $3-$dimensional smooth quadric $\widetilde{Q};$ $\mathbb{P}^3$ is the quotient of $\widetilde{Q}$ by an involution  fixing $Q.$  For any line $l$ in $\mathbb{P}^3$ the inverse image $\widetilde{l}$ is a ``straight line'' of $\widetilde{Q},$ that is a linear section of  $\widetilde{Q}$ stable under the involution. We have a double cover
$\widetilde{l}\lra l$ ramified at $l\cap Q.$ Suppose that $l$ is transversal to $Q$ and choose an orientation of $l$ and assume that $l\cap Q=\{P_{12}, P_{21}\}$. Then
\[
l\smallsetminus (l\cap{Q})\cong \mathbb{P}^1\smallsetminus\{0,\infty\}
\] 
becomes a $\mathbb{C}^\times-$principal homogeneous space. Consider a pair of points $p_1, p_2\in l\smallsetminus \{l\cap Q\}.$ As $\widetilde{l}\lra l$ is isomorphic to $z \longmapsto z^2\colon \mathbb{P}^1\lra \mathbb{P}^1,$ for $\widetilde{p}_1$ above $p_1$  and $\widetilde{p}_2$ above $p_2$ we have
\[
[\widetilde{p}_1,P_{12},\widetilde{p}_2,P_{21}]^2=[p_1,P_{12},p_2,P_{21}]
\]
and replacing $\widetilde{p}$ with the other preimage of $p$ changes $[\widetilde{p}_1,P_{12},\widetilde{p}_2,P_{21}]$ to $-[\widetilde{p}_1,P_{12},\widetilde{p}_2,P_{21}].$ Same for $p_2.$  

\begin{lemma}\label{LemmaProjectiveTriangle}
Consider a triple of distinct points $p_1, p_2,p_3\in \mathbb{P}^3\smallsetminus Q.$  Suppose that lines $(p_ip_j)\cap Q=\{P_{ij},P_{ji}\}$ are oriented. Then for  $\widetilde{p}_1, \widetilde{p}_2, \widetilde{p}_3$ lying over $p_1, p_2, p_3$ we have
\[
[\widetilde{p}_1,P_{12},\widetilde{p}_2,P_{21}][\widetilde{p}_2,P_{23},\widetilde{p}_3,P_{32}][\widetilde{p}_1,P_{13},\widetilde{p}_3,P_{31}]=\bigl[p_3,P_{31},p_1,(p_1p_3)\cap (P_{12}P_{23})\bigr]_{(p_1p_3)}.
\]
\end{lemma}
\begin{proof}
Equality 
\begin{align*}
	&\left ([\widetilde{p}_1,P_{12},\widetilde{p}_2,P_{21}][\widetilde{p}_2,P_{23},\widetilde{p}_3,P_{32}][\widetilde{p}_1,P_{13},\widetilde{p}_3,P_{31}]\right)^2\\
	=&[p_1,P_{12},p_2,P_{21}][p_2,P_{23},p_3,P_{32}][p_1,P_{13},p_3,P_{31}]\\
	=&\bigl[p_3,P_{31},p_1,(p_2p_3)\cap (P_{12}P_{23})\bigr]_{(p_2p_3)}^2	
\end{align*}
is a version of the Menelaus's theorem and can be easily checked directly, so we have 
\[
[\widetilde{p}_1,P_{12},\widetilde{p}_2,P_{21}][\widetilde{p}_2,P_{23},\widetilde{p}_3,P_{32}][\widetilde{p}_1,P_{13},\widetilde{p}_3,P_{31}]=\pm\bigl[p_3,P_{31},p_1,(p_2p_3)\cap (P_{12}P_{23})\bigr]_{(p_2p_3)}.
\]
To fix the sign, consider a case, when $p_1, p_2, p_3$ lie in $\mathbb{H}^3,$ see \S \ref{SectionIntroductionTetrahedra}. In this case we easily see that numbers
\[
[\widetilde{p}_1,P_{12},\widetilde{p}_2,P_{21}], \ [\widetilde{p}_2,P_{23},\widetilde{p}_3,P_{32}],\ [\widetilde{p}_1,P_{13},\widetilde{p}_3,P_{31}], \ \bigl[p_3,P_{31},p_1,(p_1p_3)\cap (P_{12}P_{23})\bigr]_{(p_1p_3)}
\]
are positive. From here the statement follows.
\end{proof}

Consider a marked projective tetrahedron $T=(Q,\mathcal{H}).$ Denote by $\widetilde{A}_{i}\in \widetilde{Q}$ lifts of its vertices.   Next consider a map
$\widetilde{\textup{L}}_T\colon \left ( \frac{1}{2}\mathbb{Z}\right )^{S_I} \lra \mathbb{C}^{\times}$
defined by the rule 
\begin{align*}
&\widetilde{\textup{L}}_T\left(\frac{1}{2}e_{ij}\right)=\bigl[\widetilde{A}_i,E_{ij},\widetilde{A}_j,E_{ji}\bigr] \text{ for }1\leq i<j\leq 4,\\  
&\widetilde{\textup{L}}_T\left(\frac{1}{2}e_{\varnothing}\right)=\widetilde{\textup{L}}_T\left(\frac{1}{2}e_{I}\right)=1.
\end{align*}
Denote  $\textup{L}_T$ the restriction of $\widetilde{\textup{L}}_T$ to 
$
\textup{Q}\bigl (\textup{E}_7^{\textup{L}}\bigr)\subset  \textup{Q}\bigl (\textup{E}_8 \bigr)\subset  \bigl( \frac{1}{2}\mathbb{Z}\bigr )^{S_I}.
$ 

\begin{lemma}
The function $\textup{L}_T\colon \textup{Q}\bigl(E_7^\textup{L}\bigr)\lra \mathbb{C}^\times$ does not depend on the choice of the lifts $\widetilde{A}_i\in \widetilde{Q}.$
\end{lemma}
\begin{proof} Consider a linear map $p \colon \left(\mathbb{Z}/2\mathbb{Z}\right)^{S_I}\lra S$	 sending $e_s$ to $s$ (we view $S_I$ as a vector space  with  $e_\varnothing$ as an origin). Then  $\textup{Q}\bigl(\textup{E}_7^\textup{L}\bigr)$ is contained in $\textup{Ker}(p),$ as one can easily check.  For $r\in \textup{Ker}(p)$ the value $\widetilde{\textup{L}}_T(r)$ is a product of expressions $\bigl[\widetilde{A}_i,E_{ij},\widetilde{A}_j,E_{ji}\bigr]^{\pm 1} $ such that  any lift $\widetilde{A}_i$ occurs an even number of times and thus $\widetilde{\textup{L}}_T(r)$ does not depend on the choice of lifts $\widetilde{A}_i\in \widetilde{Q}.$
\end{proof}

\begin{definition} The function $\textup{L}_T\colon \textup{Q}\bigl(\textup{E}_7^\textup{L}\bigr)\lra \mathbb{C}^\times$ is called the {\it length function} of the marked projective tetrahedron $T$. The function 
$\textup{A}_T\colon \textup{Q}\bigl(\textup{E}_7^\textup{A}\bigr)\lra \mathbb{C}^\times$ defined by $A_T=\textup{L}_{T^\vee} \circ D$ is called the {\it angle function} of $T.$ 
\end{definition}
It is easy to see that 
\[
\textup{A}_T(e_{ij})=[H_k^\vee, \widetilde{E}_{ij},H_l^\vee,\widetilde{E}_{ji}]
\]
for $k,l\in I$ such that permutation 
$\begin{pmatrix}
	1 & 2 & 3& 4\\
	i & j& k & l
\end{pmatrix}
$ is even.
 We denote 
\begin{align*}
&\textup{L}_{123}\ =\textup{L}_T\left(\frac{e_{12}+e_{13}+e_{23}+e_\varnothing}{2}\right),&
\textup{A}_{123}\ =\textup{A}_T\left(\frac{e_{12}+e_{13}+e_{23}+e_I}{2}\right),\\
&\textup{L}_{124}\ =\textup{L}_T\left(\frac{e_{12}+e_{14}+e_{24}+e_\varnothing}{2}\right),&
\textup{A}_{124}\ =\textup{A}_T\left(\frac{e_{12}+e_{14}+e_{24}+e_I}{2}\right),\\
&\textup{L}_{134}\ =\textup{L}_T\left (\frac{e_{13}+e_{14}+e_{34}+e_\varnothing}{2}\right),&
\textup{A}_{134} \ = \textup{A}_T\left (\frac{e_{13}+e_{14}+e_{34}+e_I}{2}\right),\\
&\textup{L}_{234}\ = \textup{L}_T\left (\frac{e_{23}+e_{24}+e_{34}+e_\varnothing}{2}\right),& 
\textup{A}_{234}\ =\textup{A}_T\left (\frac{e_{23}+e_{24}+e_{34}+e_I}{2}\right),\\
&\textup{L}_{1234}=\textup{L}_T\left (\frac{e_{12}+e_{14}+e_{23}+e_{34}}{2}\right),&
\textup{A}_{1234}=\textup{A}_T\left (\frac{e_{12}+e_{14}+e_{23}+e_{34}}{2}\right),\\
&\textup{L}_{1324}=\textup{L}_T\left (\frac{e_{13}+e_{14}+e_{23}+e_{24}}{2}\right),&
\textup{A}_{1324}=\textup{A}_T\left (\frac{e_{13}+e_{14}+e_{23}+e_{24}}{2}\right),\\
&\textup{L}_{1243}=\textup{L}_T\left (\frac{e_{12}+e_{13}+e_{24}+e_{34}}{2}\right),& 
\textup{A}_{1243}=\textup{A}_T\left (\frac{e_{12}+e_{13}+e_{24}+e_{34}}{2}\right).
\end{align*}
Now we have defined all notions involved in the formulation of Theorem \ref{TheoremProjectiveTetrahedra}.
\begin{definition} \label{DefinitionGeneric}
A projective tetrahedron $T$ is called {\it generic} if the only roots $r\in \textup{R}\bigl(\textup{E}_7^\textup{L}\bigr)$  such that  $\textup{L}_T(r)=1$  are $\pm e_\varnothing.$
\end{definition}

\subsection{Moduli space of generic marked projective tetrahedra} \label{SectionModuliTetrahedra}
For a function $\widetilde{\textup{L}}\colon \left ( \frac{1}{2}\mathbb{Z}\right )^{S_I} \lra \mathbb{C}^{\times}$ such that $\widetilde{\textup{L}}(e_\varnothing)=\widetilde{\textup{L}}(e_I)=1$ consider the determinant 
\be \label{FormulaDeterminant}
\det\bigl(\widetilde{\textup{L}}\bigr)=
\begin{vmatrix}
1 & \dfrac{\widetilde{\textup{L}}(\frac{e_{12}}{2})+\widetilde{\textup{L}}(-\frac{e_{12}}{2})}{2} & \dfrac{\widetilde{\textup{L}}(\frac{e_{13}}{2})+\widetilde{\textup{L}}(-\frac{e_{13}}{2})}{2} & \dfrac{\widetilde{\textup{L}}(\frac{e_{14}}{2})+\widetilde{\textup{L}}(-\frac{e_{14}}{2})}{2} \\
 \dfrac{\widetilde{\textup{L}}(\frac{e_{12}}{2})+\widetilde{\textup{L}}(-\frac{e_{12}}{2})}{2}  & 1 & \dfrac{\widetilde{\textup{L}}(\frac{e_{23}}{2})+\widetilde{\textup{L}}(-\frac{e_{23}}{2})}{2} &\dfrac{\widetilde{\textup{L}}(\frac{e_{24}}{2})+\widetilde{\textup{L}}(-\frac{e_{24}}{2})}{2} \\
 \dfrac{\widetilde{\textup{L}}(\frac{e_{13}}{2})+\widetilde{\textup{L}}(-\frac{e_{13}}{2})}{2}  & \dfrac{\widetilde{\textup{L}}(\frac{e_{23}}{2})+\widetilde{\textup{L}}(-\frac{e_{23}}{2})}{2}  & 1& \dfrac{\widetilde{\textup{L}}(\frac{e_{34}}{2})+\widetilde{\textup{L}}(-\frac{e_{34}}{2})}{2} \\
\dfrac{\widetilde{\textup{L}}(\frac{e_{14}}{2})+\widetilde{\textup{L}}(-\frac{e_{14}}{2})}{2}  & \dfrac{\widetilde{\textup{L}}(\frac{e_{24}}{2})+\widetilde{\textup{L}}(-\frac{e_{24}}{2})}{2}  & \dfrac{\widetilde{\textup{L}}(\frac{e_{34}}{2})+\widetilde{\textup{L}}(-\frac{e_{34}}{2})}{2} & 1\\
\end{vmatrix}.
\ee

\begin{lemma}\label{LemmaDeterminant}
Determinant $\det\bigl(\widetilde{\textup{L}}\bigr)$ depends only on the restriction $\textup{L}$ of $\widetilde{\textup{L}}$ to $\textup{Q}\bigl(\textup{E}_7^\textup{L}\bigr)\subset \left ( \frac{1}{2}\mathbb{Z}\right )^{S_I};$ we denote it by  $\det(\textup{L}).$ We have
\[
\det(\textup{L})=\det(\textup{L}\circ w)
\]
for any  $w\in W(\textup{D}_6).$
\end{lemma}
\begin{proof}
Expanding the determinant (\ref{FormulaDeterminant}), we obtain an explicit formula
\[
\det(\widetilde{\textup{L}})=-\frac{3}{2}+\frac{1}{2^3}\sum_{r\in \textup{R}(E_7^\textup{L})\smallsetminus \textup{R}(D_6)}\widetilde{\textup{L}}(r)-\frac{1}{2^3}\sum_{r\in \textup{R}(D_6)}\widetilde{\textup{L}}(r)+\frac{1}{2^6}\sum_{r\in \textup{R}(E_8)\smallsetminus (\textup{R}(E_7^\textup{L}) \cup \textup{R}(E_7^\textup{A}) )}\widetilde{\textup{L}}(2r).
\]
A root $r\in \textup{R}(E_8)\smallsetminus (\textup{R}(E_7^\textup{L}) \cup \textup{R}(E_7^\textup{A}) )$ is equal to $\dfrac{\pm e_{ij}\pm e_{kl} \pm e_{\varnothing} \pm e_I}{2}$ for $\{i,j\}\cup\{k,l\}=I,$  so 
\[
\widetilde{\textup{L}}(2r)=\textup{L}(\pm e_{ij})\textup{L}(\pm e_{kl}).
\] 
From here the first statement of the lemma follows. The second statement follows from the fact that $w\in W(\textup{D}_6)\subset W(\textup{E}_8)$ leaves both sets $\textup{R}\bigl(\textup{E}_7^\textup{L}\bigr)$ and $\textup{R}\bigl(\textup{E}_7^\textup{A}\bigr)$ invariant. 
\end{proof}

Consider a quasi-affine subset $\mathbb{T}$ of $\textup{Hom}\bigl(\textup{Q}\bigl(\textup{E}_7^\textup{L}\bigr),\mathbb{C}^\times\bigr)$ consisting of maps $\textup{L}$ such that $\det(\textup{L})\neq 0$ and  $\textup{L}(r)=1$ for a root $r\in \textup{R}\bigl(\textup{E}_7^{\textup{L}}\bigr)$ if and only if $r=\pm e_\varnothing.$

\begin{proposition}\label{PropositionModuliTetrahedra}
The set of isometry classes of generic marked projective tetrahedra $\textup{M}_{tetr}$ has a structure of an algebraic variety, which is an unramified double cover of  $\mathbb{T}.$ 
\end{proposition}
\begin{proof}
 Consider a map $\textup{M}_{tetr}\lra  \textup{Hom}\bigl(\textup{Q} \bigl(\textup{E}_7^\textup{L}\bigr),\mathbb{C}^\times\bigr)$ sending a tetrahedron $T$ to its length function $\textup{L}_T.$ Fix homogeneous coordinates in $\mathbb{P}^3$ so that 
\[
A_1=[1,0,0,0],\ A_2=[0,1,0,0],\  A_3=[0,0,1,0],\ A_4=[0,0,0,1].
\]
Let $q=\sum a_{ij}x_i x_j$ be an equation of $Q.$ Since $A_i \not\in Q$ we have $a_{ii}\neq 0.$ After a change of coordinates, we can assume that $a_{ii}=1$ for $\ i \in \{1,2,3,4\}.$  Then $\det\bigl(\textup{L}_T\bigr)$ is the determinant of the matrix of $Q$ and so  $\det\bigl(\textup{L}_T\bigr)\neq 0,$ because $Q$ is smooth. It follows that  $\textup{L}_T\in \mathbb{T}.$

Given a point $\textup{L}\in \mathbb{T}$ consider a projective tetrahedron $T$ with vertices 
\[
A_1=[1,0,0,0],\ A_2=[0,1,0,0],\  A_3=[0,0,1,0],\ A_4=[0,0,0,1].
\]
and a quadric 
\begin{align*}
	& x_1^2+ \textup{L}(e_{12}) x_2^2+\textup{L}(e_{13})x_3^2+ \textup{L}(e_{14})x_4^2\\
	-&\left(\textup{L}\left(e_{12}\right)+1\right)x_1x_2+
	\left (\textup{L}\left(e_{13}\right)+1\right)x_1x_3+
	\left(\textup{L}\left(e_{14}\right)+1\right)x_1x_4\\ 
	+&\left(\textup{L}\left(\frac{e_{12}+e_{13}+e_{23}+e_{\varnothing}}{2}\right)+\textup{L}\left ( \frac{e_{12}+e_{13}-e_{23}+e_{\varnothing}}{2}\right)\right)x_2x_3\\ 
	+&\left(\textup{L}\left(\frac{e_{12}+e_{14}+e_{24}+e_{\varnothing}}{2}\right)+\textup{L}\left(\frac{e_{12}+e_{14}-e_{24}+e_{\varnothing}}{2}\right)\right)x_2x_4\\ 
	+&\left(\textup{L}\left(\frac{e_{13}+e_{14}+e_{34}+e_{\varnothing}}{2}\right)+\textup{L}\left(\frac{e_{13}+e_{14}-e_{34}+e_{\varnothing}}{2}\right)\right)x_3x_4=0.
\end{align*}
It is easy to see that coordinates of points $E_{ij}$ could be computed explicitly from the values of $\textup{L}$ on the roots and there exists a canonical labelling of these points so that $\textup{L}_T=\textup{L}.$ Thus for every point $\textup{L}\in \mathbb{T}$ there exist exactly two marked projective tetrahedra with $\textup{L}_T=\textup{L}$ with different orientations of $Q$.
From here the statement follows.
 \end{proof}

The group $W(\textup{D}_6)$ acts on $\mathbb{T}$ and this action extends to the action on $\textup{M}_{tetr}.$  The stabilizer of the set of twelve roots $\pm e_{ij}$  is isomorphic to $\left(\mathbb{Z}/2\mathbb{Z}\right)^6 \rtimes \mathbb{S}_4$ and acts by changing the marking of a generic projective tetrahedron. The  group $W(\textup{D}_6)$ is generated by this subgroup and a Regge symmetry.

\section{Rational elliptic surfaces and period maps}\label{SecS}
\subsection{Rational elliptic surfaces}\label{SectionRationalElliptic}

A rational elliptic surface is a smooth complete rational surface $X,$ which admits an elliptic fibration, see \cite{HL02}, \cite{Man86},  \cite{SS10}. The elliptic fibration on $X$ is unique and is given by the anti-canonical linear system $|-K_X|$. We assume that the elliptic fibration is relatively minimal and has a section. It is known that every rational elliptic surface can be obtained by blowing up the nine base points of a pencil of plane cubic curves having at least one smooth member.  

The Picard lattice $\textup{Pic}(X)$ has rank $10$ and signature $(1,9).$ Denote by $f$ the class of a fiber, which is equal to $-K_{X}.$ The orthogonal complement $f^{\perp}\subset \textup{Pic}(X)$ contains an isotropic vector $f$ and is an affine root lattice of type $E_8^1$. The quotient $f^{\perp}/f$  is a lattice of type $E_8$; we have a projection 
\[
\pi\colon f^{\perp}  \lra f^{\perp}/f.
\] 
An element $r \in \textup{Pic}(X)$ is called a root if $r\in f^{\perp} $ and $r^2=-2.$ If $s_0$ is a section of $X$ then the lattice $\langle s_0, f \rangle$ is unimodular so
\[
\textup{Pic}(X)=\langle s_0, f \rangle \oplus \langle s_0, f \rangle^{\perp}.
\]
The map $\langle s_0, f \rangle^{\perp} \lra f^{\perp}/ f$ is an isometry. 

From the adjunction formula it follows that the self-intersection number of a smooth rational curve $C$ on a rational elliptic surface is greater or equal than $-2.$ Smooth rational curves with self-intersection number~$-1$ are  sections of the elliptic fibration. Smooth rational curves with self-intersection number~$-2$ are irreducible components of reducible fibers of the elliptic fibration. The classification of singular fibers of an elliptic fibration goes back to Kodaira. The Euler characteristic of a rational elliptic surface is equal to~$12,$ so it can have at most twelve singular fibers. The simplest type of a singular fiber is called $I_n;$ such fiber is a ``wheel'' made up of $n$ smooth rational curves intersecting transversally.

\begin{definition}
Let $X$ be a rational elliptic surface with a pair $F_1, F_2$ of singular fibers of type $I_2$.   	We call a triple $(X, F_1, F_2)$ a  {\it $D_6$-surface}. A $D_6$-surface is called {\it} generic if $F_1$ and $F_2$ are the only reducible fibers of the elliptic fibration.
\end{definition}

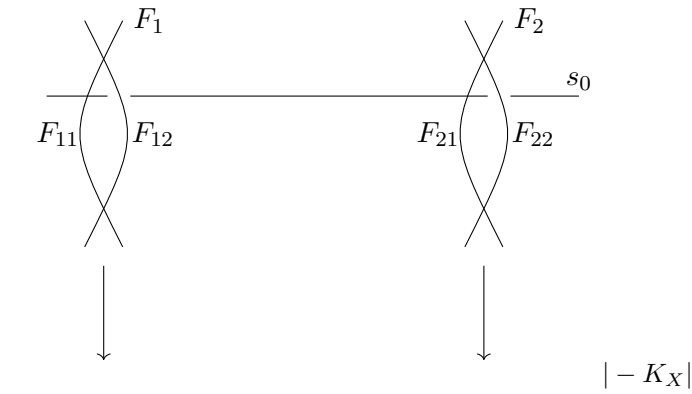
\begin{figure}
\centering
\begin{tikzpicture}
\draw (2,1) .. controls (2.75,2.5) and (2.75,2.5) .. (2,4);
\draw (2.5,1) .. controls (1.75,2.5) and (1.75,2.5) .. (2.5,4);
\draw (7,1) .. controls (7.75,2.5) and (7.75,2.5) .. (7,4);
\draw (7.5,1) .. controls (6.75,2.5) and (6.75,2.5) .. (7.5,4);
\draw (1,-1) -- (9,-1);
\draw (1.5,3) -- (2.3,3);
\draw (2.6,3) -- (7.3,3);
\draw (7.6,3) -- (8.5,3);
\draw[->] (2.25,0.75) -- (2.25,-0.5);
\draw[->] (7.25,0.75) -- (7.25,-0.5);
\draw (2.85,4) node{$F_1$};
\draw (2.9,2.5) node{$F_{12}$};
\draw (1.65,2.5) node{$F_{11}$};
\draw (7.85,4) node{$F_2$};
\draw (7.9,2.5) node{$F_{22}$};
\draw (6.65,2.5) node{$F_{21}$};
\draw (8.5,3.2) node{$s_0$};
\draw (9.4,-0.7) node{\mbox{\small$|-K_X|$}};
\end{tikzpicture}
\caption{The elliptic fibration on a $D_6$-surface with fibers $F_1$ and $F_2$ of type $I_2$.}
\label{FigureEllipticFibration}
\end{figure}

The term ``$D_6$-surface'' is justified by the fact that the orthogonal complement in  $f^{\perp}/f$ to the classes of the components of the fibers $F_1$ and $F_2$ is a root lattice of type $D_6.$ Recall that in \S \ref{SectionE8RootSystem} we defined a root system $\textup{E}_8.$

\begin{definition} \label{DefinitionMarking}
A marking of a $D_6$-surface $X$ consists of the following data.
\begin{enumerate}
\item A choice of a section $s_0$ of the elliptic  fibration (zero section).
\item An isometry $m\colon \textup{Q}(\textup{E}_8)\lra f^{\perp}/f$ sending $e_I$ and $e_{\varnothing}$ to the classes of components of fibers $F_1$ and $F_2,$ which intersect $s_0.$
\item A choice of a nodal point in $F_1$ and in $F_2$.
\end{enumerate}
The Weyl group $W(\textup{D}_6)$ acts on the set of markings by changing the isomorphism $m.$
\end{definition}

We will usually omit $m$ from the notation and simply put $\textup{Q}(\textup{E}_8)=f^{\perp}/f.$	We denote the components of $F_1$ and $F_2$ intersecting $s_0$ by  $F_{11}$ and $F_{21}.$ The other components are denoted  $F_{12}$ and $F_{22},$ see Figure~\ref{FigureEllipticFibration}. Thus
\begin{align*}
\pi\left(\left[F_{11}\right]\right)&=e_I,\\ 
\pi\left(\left[F_{12}\right]\right)&=-e_I,\\
\pi\left(\left[F_{21}\right]\right)&=e_{\varnothing},\\ 
\pi\left(\left[F_{22}\right]\right)&=-e_{\varnothing}.
\end{align*}
The first nodal point of $F_1$ is denoted by $F_{1}^x,$ the second is denoted by $F_{1}^y$, and similarly for $F_2$.

\subsection{Period maps}\label{secL}
In this section we adapt the ideas of \cite[\S 5]{Loo81} to the case of a marked $D_6$-surface $X$. A choice of a nodal point in $F_1$ and $F_2$ fixes isomorphisms 
\[
\textup{Pic}^0(F_1)\cong  \C^{\times}, \ \textup{Pic}^0(F_2) \cong \C^{\times}.
\] 
Let $d\in \textup{Pic}(X)$ be a line bundle, which restricts trivially to both components of $F_1.$ In this case $d\in f^{\perp}$ and $\pi(d)\in \textup{Q}\bigl(\textup{E}_7^\textup{L}\bigr).$ The restriction of $d$ to $F_1$ determines an element $\textup{Res}_{F_1}(d)\in\textup{Pic}^0(F_1).$ Since the restriction of $f\in f^{\perp}$ is trivial, the image of $d$ in $\textup{Pic}^0(F_1)$ depends only on  $\pi(d)\in \textup{Q}\bigl(E_7^\textup{L}\bigr).$  Thus we have constructed a {\it period map}
\[
\textup{Res}_{F_1} \colon \textup{Q}\bigl(\textup{E}_7^\textup{L}\bigr) \lra \C^{\times}.
\]
Similarly, by restricting to $F_2$ we define a second period map
\[
\textup{Res}_{F_2} \colon \textup{Q}\bigl(\textup{E}_7^\textup{A}\bigr)\lra \C^{\times}.
\]

Here is a more explicit description of the map $\textup{Res}_{F_1}$ (similar statements hold for $\textup{Res}_{F_2}$). Consider a root $r\in \textup{Pic}(X),$ which can be represented as a difference of classes of sections $s_1$ and $s_2$ intersecting the fiber $F_1$ in the same component. If this component is $F_{11},$ then
\begin{align}\label{FormulaExplicitPeriod}
\textup{Res}_{F_1}(r)=\textup{Res}_{F_1} \left([s_1]-[s_2]\right)=\left[F_{1}^x,s_1 \cap F_{11}, F_{1}^y, s_2 \cap F_{11}\right]_{F_{11}}.
\end{align}
If this component is  $F_{12},$ then
\begin{align}\label{FormulaExplicitPeriod2}
\textup{Res}_{F_1}(r)=\textup{Res}_{F_1}\left([s_1]-[s_2]\right)=\left[F_{1}^y,s_1 \cap F_{12}, F_{1}^x, s_2 \cap F_{12}\right]_{F_{12}}.
\end{align}

\begin{lemma} \label{LemmaGenericSurface}
Let $X$ be a $D_6-$surface. The following conditions are equivalent.
\begin{enumerate}
\item $X$ is generic.
\item We have $\textup{Res}_{F_1}(r)=1$ for a root $r \in R\bigl(\textup{E}_7^\textup{L}\bigr)$ if and only if $r=\pm e_{\varnothing}.$
\item We have $\textup{Res}_{F_2}(r)=1$ for a root  $r \in R\bigl(\textup{E}_7^\textup{A}\bigr)$ if and only if $r=\pm e_{I}.$
\end{enumerate}
\end{lemma}
\begin{proof}
	Let $D$ be an irreducible component of a reducible fiber, distinct from $F_1$ and $F_2.$ Then $\pi([D])\in \textup{R}(\textup{D}_6)$ and we have
\[
\textup{Res}_{F_1}([D])=\textup{Res}_{F_2} ([D])=1.
\]
This proves that {\it 2} implies {\it 1} and that  {\it 3} implies {\it 1}. To show that {\it 1} implies {\it 2} assume that there exists a root $r\in \textup{R}\bigl(\textup{E}_7^\textup{L}\bigr)$ such that  $\textup{Res}_{F_1}(r)=1$ and $r\neq \pm e_{\varnothing}$.  By \cite[Proposition 5.2]{Loo81}  there exists a component $D$ of a reducible fiber $F\neq F_1, F_2$ of the elliptic fibration such that $r=\pi([D]).$ Thus $X$ is not generic. Similarly,  {\it 1} implies {\it 3}.
\end{proof}

\subsection{Admissible conic bundles}\label{SectionConicBundle}

\begin{definition}
A conic bundle on a $D_6$-surface $X$ is a surjective morphism
$
b\colon X \lra \mathbb{P}^1
$
such that the generic fiber is a smooth rational curve. The conic bundle  $b$ is called {\it admissible} if the irreducible components  of fibers $F_1$ and $F_2$  are sections of $b.$ The fiber $b^{-1}(p)$ will be denoted $b_p$ and its class $[b_{p}] \in \textup{Pic}(X)$ will be denoted by $B.$ 
\end{definition}

\begin{lemma}
	Every singular fiber of an admissible conic bundle on a $D_6$-surface $X$ is a chain of smooth rational curves 
	$Y_1,  \ldots, Y_k$ for $k\geq 2,$ where $Y_1, Y_k$ are $(-1)$-curves and $Y_2, \ldots, Y_{k-1}$  are $(-2)-$curves. 
\label{fib}
\end{lemma}

\begin{proof}
This result follows immediately from the classification of singular fibers of conic bundles on rational elliptic surfaces, see \cite[Proposition $5.1$]{GS19}. 
\end{proof}

Assume that $b\colon X\lra \mathbb{P}^1$ is an admissible conic bundle on $X.$ Fix a zero section $s_0$ on $X$ and label the components of fibers $F_1$ and $F_2$ as in \S \ref{SectionRationalElliptic}. A computation of Euler characteristic shows that a conic bundle $b\colon X\lra \mathbb{P}^1$ has at most eight singular fibers. 
 Since curves  $F_{11}, F_{12},F_{21},F_{22}$ are sections of the conic bundle, for every singular fiber $b_p=Y_1\cup \ldots \cup Y_k$  the points $Y_1\cap F_1$ lie in different components of $F_1$ and the points $Y_1\cap F_2$ lie in different components of $F_2$; similarly for $Y_2$. We denote the component of $b_p$ intersecting $F_{11}$ by $b_p^{1}$ and the component intersecting $F_{21}$ by $b_p^{2}.$ Notice that $b_p^1$ and $b_p^{2}$ are not necessarily distinct.

\begin{definition}
The conic bundle function 
$C_{X,b}\in \mathbb{C}(\mathbb{P}^1)$ is a rational function on the target space of $b$ which associates to a point $p\in \mathbb{P}^1$ the cross-ratio of the four points in which the fiber $b_p$ intersects $F_1$ and $F_2,$ namely
\begin{align}\label{FormulaConicBundle}
C_{X,b}(p)=\left[b_p\cap F_{11},b_p\cap F_{21},b_p\cap F_{12},b_p\cap F_{22}\right]_{b_p}.	 
\end{align}
\end{definition}

\begin{lemma} 
\begin{enumerate}
\item $C_{X,b}\in \mathbb{C}(\mathbb{P}^1)$ is a rational function of degree $4.$ 
\item  Assume that fiber $b_p$ has $k$ irreducible components. Then
\[
\textup{ord}_p(C_{X,b})=
\begin{cases}
0 & \text{ if } k=1,\\	
k-1 & \text{ if } k>1 \text{ and }b_p^1=b_p^2,\\	
-(k-1)& \text{ if }k>1 \text{ and }b_p^1\neq b_p^2.\\	
\end{cases}
\]  
\item $C_{X,b}$ takes the value $1$ at four points 
\[
b\left(F_{1}^x\right), b\left(F_{1}^y\right), b\left(F_2^x\right), b\left(F_{2}^y\right)\in \mathbb{P}^1.
\] 
\end{enumerate}
\label{LemmaConicBundleFunction}  
\end{lemma}
\begin{proof}
This statement follows directly from (\ref{FormulaConicBundle}) and Lemma \ref{fib}. 	
\end{proof}

Assume that 
\begin{align}\label{EqualityDivisorConic}
\textup{div}\left(C_{X,b}\right)=p_1+p_2+p_3+p_4-p_5-p_6-p_7-p_8
\end{align}
for points $p_1,\ldots,p_8\in \mathbb{P}^1,$ which are not necessarily distinct.  Then $b_{p_i}^1=b_{p_i}^2$ for $i\in\{1,2,3,4\}$ and $b_{p_i}^1\neq b_{p_i}^2$ for $i\in\{5,6,7,8\}.$

 Our next goal is to write down the conic bundle function explicitly. For that we  choose a coordinate on  $\mathbb{P}^1.$ There are two natural ways to do it. Consider coordinates $t_1, t_2$ on $\mathbb{P}^1$ such that
\begin{align*}
&t_1(b(F_{1}^x))=t_2(b(F_{2}^x))=0,\\ 
&t_1(b(F_{1}^y))=t_2(b(F_{2}^y))=\infty,\\ 
&t_1(p_8)=t_2(p_8)=1.
\end{align*}
 Denote by  $\psi_{X,b}\in \textup{PGL}_2(\mathbb{C})$ the M\"obius transformation such that $t_2=\psi_{X,b}(t_1).$  
 \begin{lemma}
For $1\leq i\leq 8$ we have $t_1\left( p_i\right)=\textup{Res}_{F_1}\left([b_{p_i}^1-b_{p_8}^1]\right)$ and $t_2\left( p_i\right)= \textup{Res}_{F_2}\left([b_{p_i}^2-b_{p_8}^2]\right).$ 
\end{lemma}
\begin{proof}
	The morphism $b\colon X \lra \mathbb{P}^1$ induces an isomorphism between $F_{11}$ and $\mathbb{P}^1,$ so the conic bundle function can be viewed as a rational function on $F_{11}.$ The first statement follows from Lemma \ref{LemmaConicBundleFunction} and  (\ref{FormulaExplicitPeriod}). The proof of the second statement is similar.
\end{proof}

\begin{corollary}\label{CorollaryConicFunction} The  function $C_{X,b}\in \mathbb{C}(\mathbb{P}^1)$ can be written as a rational function of $t_1$ or as a rational function of~$t_2:$ 
\[
C_{X,b}=C^{\textup{L}}_{X,b}(t_1)=C^{\textup{A}}_{X,b}(t_2).
\]
We have
\begin{align*}
&C_{X,b}^{\textup{L}}(t)= \frac{\left (t-\textup{Res}_{F_1}([b_{p_1}^1-b_{p_8}^1]) \right)\left (t-\textup{Res}_{F_1}([b_{p_2}^1-b_{p_8}^1]) \right)\left (t-\textup{Res}_{F_1}([b_{p_3}^1-b_{p_8}^1]) \right)\left (t-\textup{Res}_{F_1}([b_{p_4}^1-b_{p_8}^1]) \right)}{\left (t-\textup{Res}_{F_1}([b_{p_5}^1-b_{p_8}^1]) \right)\left (t-\textup{Res}_{F_1}([b_{p_6}^1-b_{p_8}^1]) \right)\left (t-\textup{Res}_{F_1}([b_{p_7}^1-b_{p_8}^1]) \right)\left (t-\textup{Res}_{F_1}([b_{p_8}^1-b_{p_8}^1]) \right)},\\
&C_{X,b}^{\textup{A}}(t)=\frac{\left (t-\textup{Res}_{F_2}([b_{p_1}^2-b_{p_8}^2]) \right)\left (t-\textup{Res}_{F_2}([b_{p_2}^2-b_{p_8}^2]) \right)\left (t-\textup{Res}_{F_2}([b_{p_3}^2-b_{p_8}^2]) \right)\left (t-\textup{Res}_{F_2}([b_{p_4}^2-b_{p_8}^2]) \right)}{\left (t-\textup{Res}_{F_2}([b_{p_5}^2-b_{p_8}^2]) \right)\left (t-\textup{Res}_{F_2}([b_{p_6}^2-b_{p_8}^2]) \right)\left (t-\textup{Res}_{F_2}([b_{p_7}^2-b_{p_8}^2]) \right)\left (t-\textup{Res}_{F_2}([b_{p_8}^2-b_{p_8}^2]) \right)}.
\end{align*}
We have $C_{X,b}^{\textup{L}}(t)=C_{X,b}^{\textup{A}}(\psi_{X,b}(t)).$
\end{corollary}

\subsection{Conic bundles on a generic $D_6$-surface}\label{SectionConicBundleGeneric}

\begin{proposition}
If $X$ is a generic $D_6$-surface then there exists an admissible conic bundle $b\colon X\lra \mathbb{P}^1$.	
\end{proposition}
\begin{proof}
By \cite[Theorem 3.5]{Fus06} a generic $D_6$-surface $X$ is obtained as a blow up of $\mathbb{P}^2$ at the nine base points of a pencil of plane cubics generated by curves $C_1 \cup l_1$ and  $C_2 \cup l_2$ for conics  $C_1, C_2$  and lines $l_1, l_2$ such that $C_1\cap C_2 \cap l_1 =\varnothing$ and $C_1\cap C_2 \cap l_2 =\varnothing.$ Pick a point $p\in  C_1\cap C_2.$ The pencil of lines passing through $p$ defines an admissible conic bundle on $X.$ 
\end{proof}

Lemma \ref{fib} implies that $b$ has exactly eight singular fibers, because otherwise a component $Y_2$ of a singular fiber would be a component of a reducible fiber $F$ of the elliptic fibration, such that $F$ is distinct from $F_1$ and $F_2$. It follows that points $b_1,\ldots,b_8$ are distinct. Let us fix their order and assume that (\ref{EqualityDivisorConic}) holds.

We are going to define a marking of the surface $X.$ Since we have already chosen a zero section and ordered singular points of the fibers, we just need to construct an isometry between $f^{\perp}/f$ and $\textup{Q}(\textup{E}_8).$ For this we choose a base of roots in $f^\perp/f$ and in $E_8$ as in Figure \ref{FigureSimpleRoots}; there exists an isometry identifying the corresponding simple roots. To check that this defines  a marking in the sense of Definition \ref{DefinitionMarking} we need to prove that $e_{\varnothing}=\pi([F_{21}]);$ Lemma \ref{LemmaRelationInPicard} does this.


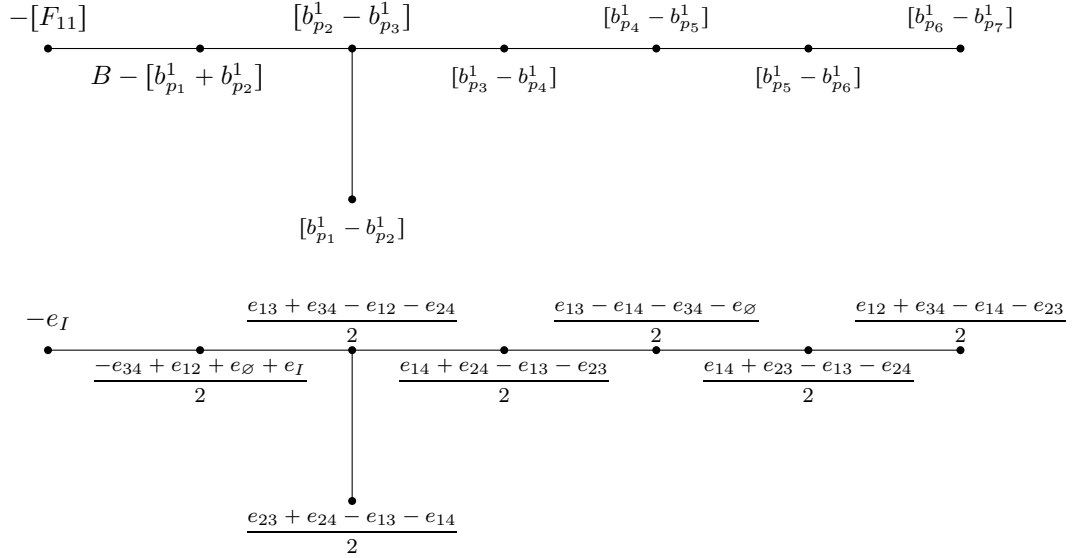
\begin{figure}
\centering
\begin{tikzpicture}

\draw (0,4) -- (12,4);
\draw (4,4) -- (4,2);

\draw (0,4) node[circle,fill=black,inner sep=0pt,minimum size=3pt]{};
\draw (0,4.4) node{\mbox{\small $-[F_{11}]$}};

\draw (2,4) node[circle,fill=black,inner sep=0pt,minimum size=3pt]{};
\draw (1.7,3.6) node{\mbox{\small $B-[b_{p_1}^1+b^1_{p_2}]$}};

\draw (4,4) node[circle,fill=black,inner sep=0pt,minimum size=3pt]{};
\draw (4,4.4) node{\mbox{\small $[b_{p_2}^1-b_{p_3}^1]$}};

\draw (4,2) node[circle,fill=black,inner sep=0pt,minimum size=3pt]{};
\draw (4,1.6) node{\mbox{\scriptsize $[b_{p_1}^1-b_{p_2}^1]$}};

\draw (6,4) node[circle,fill=black,inner sep=0pt,minimum size=3pt]{};
\draw (6,3.6) node{\mbox{\scriptsize $[b_{p_3}^1-b_{p_4}^1]$}};

\draw (8,4) node[circle,fill=black,inner sep=0pt,minimum size=3pt]{};
\draw (8,4.4) node{\mbox{\scriptsize $[b_{p_4}^1-b_{p_5}^1]$}};

\draw (10,4) node[circle,fill=black,inner sep=0pt,minimum size=3pt]{};
\draw (10,3.6) node{\mbox{\scriptsize $[b_{p_5}^1-b_{p_6}^1]$}};

\draw (12,4) node[circle,fill=black,inner sep=0pt,minimum size=3pt]{};
\draw (12,4.4) node{\mbox{\scriptsize $[b_{p_6}^1-b_{p_7}^1]$}};

\draw (0,0) -- (12,0);
\draw (4,0) -- (4,-2);

\draw (0,0) node[circle,fill=black,inner sep=0pt,minimum size=3pt]{};
\draw (0,0.4) node{\mbox{\small $-e_{I}$}};

\draw (2,0) node[circle,fill=black,inner sep=0pt,minimum size=3pt]{};
\draw (2,-0.4) node{\mbox{\scriptsize $\dfrac{-e_{34}+e_{12}+e_{\varnothing}+e_I}{2}$}};

\draw (4,0) node[circle,fill=black,inner sep=0pt,minimum size=3pt]{};
\draw (4,0.4) node{\mbox{\scriptsize $\dfrac{e_{13}+e_{34}-e_{12}-e_{24}}{2}$}};

\draw (4,-2) node[circle,fill=black,inner sep=0pt,minimum size=3pt]{};
\draw (4,-2.4) node{\mbox{\scriptsize $\dfrac{e_{23}+e_{24}-e_{13}-e_{14}}{2}$}};

\draw (6,0) node[circle,fill=black,inner sep=0pt,minimum size=3pt]{};
\draw (6,-0.4) node{\mbox{\scriptsize $ \dfrac{e_{14}+e_{24}-e_{13}-e_{23}}{2}$}};

\draw (8,0) node[circle,fill=black,inner sep=0pt,minimum size=3pt]{};
\draw (8,0.4) node{\mbox{\scriptsize $\dfrac{e_{13}-e_{14}-e_{34}-e_{\varnothing}}{2}$}};

\draw (10,0) node[circle,fill=black,inner sep=0pt,minimum size=3pt]{};
\draw (10,-0.4) node{\mbox{\scriptsize $\dfrac{e_{14}+e_{23}-e_{13}-e_{24}}{2}$}};

\draw (12,0) node[circle,fill=black,inner sep=0pt,minimum size=3pt]{};
\draw (12,0.4) node{\mbox{\scriptsize $\dfrac{e_{12}+e_{34}-e_{14}-e_{23}}{2}$}};
\end{tikzpicture}

\caption{Matching root bases of lattices $f^{\perp}/f$ and $Q(E_8)$. We use the same notation for a class $r\in f^\perp$ and its projection $\pi(r)\in f^{\perp}/f.$ 
}
\label{FigureSimpleRoots}
\end{figure}

\begin{lemma}\label{LemmaRelationInPicard}  The following equalities hold in $\textup{Pic}(X):$
\[2B=\sum_{i=1}^4[b_{p_i}^1]+[F_{11}]-[F_{22}].\]
\end{lemma}
\begin{proof}
Consider eight $(-1)-$curves $b_{p_i}^1, 1\leq i \leq 8.$ These curves are pairwise disjoint, so could be blown down simultaneously by a morphism $r \colon X \lra Y.$
Rational surface $Y$  has Picard number $2$ and contains a $(-2)$-curve $r(F_{12})$, hence $Y$ is isomorphic to the Hirzenbruch surface $\Sigma_2.$ Let $s_Y=[r(F_{12})]\in \textup{Pic}(Y)$ be the class of the zero section of $Y$  and $f_Y\in \textup{Pic}(Y)$ be the classes of a fiber. It is known that $\textup{Pic}(Y)=\langle s_Y, f_Y \rangle$ and $s_Y^2=-2,\ s_Yf_Y=1,\ f_Y^2=0.$
Computing the intersection indices one can see that
$r_*([F_{11}])=s_Y+4f_Y$ and $r_*([F_{22}])=s_Y+2f_Y.$ It follows that
\begin{align*}
	&r^*r_*([F_{11}])=[F_{11}]+\sum_{i=1}^8[b_{p_i}^1],\\
	&r^*r_*([F_{22}])=[F_{22}]+\sum_{i=5}^8[b_{p_i}^1],
\end{align*}
so 
\[
\sum_{i=1}^4[b_{p_i}^1]+[F_{11}]-[F_{22}]=2r^*(f_Y)=2B.
\]
\end{proof}


From here one can deduce that the following equalities hold:
\begin{align*}
&\frac{e_{23}+e_{24}+e_{34}-e_{\varnothing}}{2}=\pi([b_{p_1}^1-b_{p_8}^1]),
&\frac{e_{12}+e_{13}+e_{14}+e_I}{2}=-\pi([b_{p_1}^2-b_{p_8}^2]),\\
&\frac{e_{13}+e_{14}+e_{34}-e_{\varnothing}}{2}=\pi([b_{p_2}^1-b_{p_8}^1]),
&\frac{e_{12}+e_{23}+e_{24}+e_I}{2}=-\pi([b_{p_2}^2-b_{p_8}^2]),\\
&\frac{e_{12}+e_{14}+e_{24}-e_{\varnothing}}{2}=\pi([b_{p_3}^1-b_{p_8}^1]),
&\frac{e_{13}+e_{23}+e_{34}+e_I}{2}=-\pi([b_{p_3}^2-b_{p_8}^2]),\\
&\frac{e_{12}+e_{13}+e_{23}-e_{\varnothing}}{2}=\pi([b_{p_4}^1-b_{p_8}^1]),
&\frac{e_{14}+e_{24}+e_{34}+e_I}{2}=-\pi([b_{p_4}^2-b_{p_8}^2]),\\
&\frac{e_{12}+e_{14}+e_{23}+e_{34}}{2}=\pi([b_{p_5}^1-b_{p_8}^1]),
&\frac{e_{12}+e_{14}+e_{23}+e_{34}}{2}=-\pi([b_{p_5}^2-b_{p_8}^2]),\\
&\frac{e_{12}+e_{13}+e_{24}+e_{34}}{2}=\pi([b_{p_6}^1-b_{p_8}^1]),
&\frac{e_{12}+e_{13}+e_{24}+e_{34}}{2}=-\pi([b_{p_6}^2-b_{p_8}^2]),\\
&\frac{e_{13}+e_{14}+e_{23}+e_{24}}{2}=\pi([b_{p_7}^1-b_{p_8}^1]),
&\frac{e_{13}+e_{14}+e_{23}+e_{24}}{2}=-\pi([b_{p_7}^2-b_{p_8}^2]).
\end{align*}


\subsection{The moduli space of generic marked $D_6$-surfaces}\label{SectionModuliSurfaces}

Let $X$ be a marked $D_6$-surface. The period map is a point $\textup{Res}_{F_1}\in \textup{Hom}\bigl(\textup{Q}\bigl(\textup{E}_7^\textup{L}\bigr),\mathbb{C}^\times\bigr).$ In \S \ref{SectionModuliTetrahedra}
 we defined an open subset $\mathbb{T}\subseteq \textup{Hom}\bigl(\textup{Q}\bigl(\textup{E}_7^\textup{L}\bigr),\mathbb{C}^\times\bigr).$ 
\begin{proposition}\label{PropositionDeterminant}
For a generic marked $D_6$-surface $(X,F_1,F_2)$ we have 
$\textup{Res}_{F_1}\in \mathbb{T}.$
\end{proposition}
\begin{proof}
From Lemma \ref{LemmaGenericSurface} we see that $\textup{Res}_{F_1}(r)=1$ for a root $r\in \textup{R}\bigl(\textup{E}_7^{\textup{L}}\bigr)$ if and only if $r=\pm e_\varnothing.$ 
It remains to prove that $\det\left(\textup{Res}_{F_1}\right)\neq 0.$ By Lemma \ref{LemmaDeterminant} this statement does not depend on the choice of the marking of $X$. It is convenient to choose the marking introduced in \S \ref{SectionConicBundleGeneric}. The points $t_1(b\left(F_2^x\right))$ and $t_1(b\left(F_{2}^y\right))$ are distinct, so the discriminant of the quadratic polynomial
\begin{align*}
	&\frac{1}{t}\Bigl(\bigl (t-\textup{Res}_{F_1}([b_{p_1}^1-b_{p_8}^1]) \bigr)\bigl (t-\textup{Res}_{F_1}([b_{p_2}^1-b_{p_8}^1]) \bigr)\bigl (t-\textup{Res}_{F_1}([b_{p_3}^1-b_{p_8}^1]) \bigr)\bigl (t-\textup{Res}_{F_1}([b_{p_4}^1-b_{p_8}^1]) \bigr)\\
	-&\bigl (t-\textup{Res}_{F_1}([b_{p_5}^1-b_{p_8}^1]) \bigr)\bigl (t-\textup{Res}_{F_1}([b_{p_6}^1-b_{p_8}^1]) \bigr)\bigl (t-\textup{Res}_{F_1}([b_{p_7}^1-b_{p_8}^1])\bigr)\bigl (t-\textup{Res}_{F_1}([b_{p_8}^1-b_{p_8}^1]) \bigr)\Bigr)
\end{align*}
is not equal to zero. One can check by a direct but tedious computation that the discriminant of a quadratic polynomial
\begin{align*}
&\frac{1}{t}\Bigl((t-a_{12}a_{23}a_{13})(t-a_{12}a_{24}a_{14})(t-a_{13}a_{34}a_{14})(t-a_{24}a_{34}a_{23})\\
-&(t-a_{12}a_{23}a_{34}a_{14})(t-a_{13}a_{23}a_{24}a_{14})(t-a_{12}a_{24}a_{34}a_{13})(t-1)\Bigr)
\end{align*}
is equal to 
\[
16\left(\prod_{i<j}a_{ij}\right)^2
\begin{vmatrix}
1 & \dfrac{a_{12}+\dfrac{1}{a_{12}}}{2}& \dfrac{a_{13}+\dfrac{1}{a_{13}}}{2}&\dfrac{a_{14}+\dfrac{1}{a_{14}}}{2} \\
 \dfrac{a_{12}+\dfrac{1}{a_{12}}}{2} & 1 &\dfrac{a_{23}+\dfrac{1}{a_{23}}}{2}&\dfrac{a_{24}+\dfrac{1}{a_{24}}}{2}\\
 \dfrac{a_{13}+\dfrac{1}{a_{13}}}{2} & \dfrac{a_{23}+\dfrac{1}{a_{23}}}{2} & 1&\dfrac{a_{34}+\dfrac{1}{a_{34}}}{2}\\
\dfrac{a_{14}+\dfrac{1}{a_{14}}}{2}  & \dfrac{a_{24}+\dfrac{1}{a_{24}}}{2}  & \dfrac{a_{34}+\dfrac{1}{a_{34}}}{2} & 1\\
\end{vmatrix},
\]
which implies that $\det(\textup{Res}_{F_1})\neq 0.$
\end{proof}

\begin{proposition}\label{PropositionModuliSurfaces} 
The set $\textup{M}_{\text{surf}}$ of isomorphism classes of generic marked $D_6$-surfaces has a structure of an algebraic variety, which is an unramified double cover of  $\mathbb{T}.$ 
\label{ModSerf}
\end{proposition}

\begin{proof}
Consider the map that associates to a marked $D_6$-surface $X$  the period map $\textup{Res}_{F_1}\in \mathbb{T}.$ From \cite[Proposition 5.5]{Loo81} it follows that this map is surjective.  

Next, consider two marked $D_6$-surfaces $X$ and $X'$ with $\textup{Res}_{F_1}=\textup{Res}_{F_1'}.$ Recall the orthogonal decompositions
\begin{align*}
&\textup{Pic}(X_1)=\langle s_0, f \rangle \oplus \bigl(f^{\perp}/f\bigr),\ \ \ \ \textup{Pic}(X_2)=\langle s_0', f' \rangle \oplus \bigl(f'^{\perp}/f'\bigr).
\end{align*}
Define an isometry $\psi \colon \textup{Pic}(X_1) \lra \textup{Pic}(X_2)$
by a rule $\psi(s_0)=s_0',\ \  \psi(f)=f'$ and $\psi = m_2 \circ m^{-1}_1$ on $f^{\perp}/f$. 
Now we are in a position to apply the Torelli theorem, see \cite[Theorem 5.3]{Loo81}. Indeed, $(X,F)$ and  $(X,F')$ are smooth rational surfaces endowed with negative oriented anti-canonical cycles $F_{11}+F_{12}$ and $F_{11}'+F_{12}'.$ The isometry $\psi$ sends roots to roots,  the positive cone to the positive cone, and  $\psi$ commutes with period maps. Since $X$ is a generic $D_6$-surface, it has only four $-2$-curves, so the lattice $\langle F_{11},F_{12}\rangle^\perp$ has two nodal roots, namely $\bigl(f,m_1(e_\varnothing)\bigr)$ and $\bigl(0,-m_1(e_\varnothing)\bigr);$ the same for $X'.$ Since $\psi\bigl(f,m_1(e_\varnothing)\bigr)=(f',m_2(e_\varnothing))$ and $\psi\bigl(0,-m_1(e_\varnothing)\bigr)=\bigl(0,-m_2(e_\varnothing)\bigr),$ isometry $\psi$ sends nodal roots to nodal roots. By \cite[Theorem 5.3]{Loo81} $\psi$ must be induced by an isomorphism $\Psi \colon X_1 \lra X_2$ such that $\Psi(F_1^x)=(F_1^{x})'$ and $\Psi(F_1^y)=(F_1^{y})'.$ The only ambiguity is in the action of $\Psi$ on the marking of nodal points of the second special fiber. From here the statement follows.
\end{proof}

\section{Correspondence between tetrahedra and  $D_6$-surfaces}\label{SecC}

\subsection{Proofs of Theorem \ref{TheoremProjectiveTetrahedra} and Theorem \ref{TheoremCorrespondence}}

In \S \ref{SectionModuliTetrahedra} we defined the moduli space of generic marked projective tetrahedra $M_{tetr}$ and in \S  \ref{SectionModuliSurfaces} we defined the moduli space of generic marked $D_6$-surfaces $M_{surf}.$ Propositions \ref{PropositionModuliTetrahedra} and \ref{PropositionModuliSurfaces} imply that both $M_{tetr}$ and $M_{surf}$ are unramified double covers of the same variety $\mathbb{T}.$ In \S \ref{SectionConstruction} we construct a rational map  
\[
Cor \colon M_{tetr} \dashrightarrow M_{surf}
\]
 sending  $T$ to  $(X_T,F_1,F_2)$. In \S \ref{SectionEqualityLengths} we show that this map commutes with projections to $\mathbb{T}.$ It is easy to see that a rational map $U\dashrightarrow V,$ which commutes with {\'e}tale maps $U\lra X$ and   $V\lra X$ can be extended to a morphism, so the rational map $Cor$ can be extended to a morphism 
 \[
 \textup{Cor}\colon M_{tetr} \lra M_{surf}.
 \]
From the construction it follows immediately that  $\textup{Cor}$ is equivariant with respect to the deck transformations of the coverings, so  $ \textup{Cor}$ is an isomorphism. In \S \ref{SectionEqualityAngles} we prove the equality $\textup{Res}_{F_2}=\textup{A}_T,$ which finishes the proof of Theorem \ref{TheoremCorrespondence}.
 
Finally, we show that Theorem \ref{TheoremCorrespondence} implies Theorem \ref{TheoremProjectiveTetrahedra}. Indeed, Cho-Kim functions of a tetrahedron (see Definition \ref{DefinitionChoKim}) coincide with conic bundle functions of the corresponding $D_6$-surface, so Theorem  \ref{TheoremProjectiveTetrahedra} follows from Corollary \ref{CorollaryConicFunction}.

\subsection{The construction of the correspondence}\label{SectionConstruction}
Our goal in this section is to construct a rational map 
\[
Cor \colon M_{tetr} \dashrightarrow M_{surf}.
\]
 Since we are constructing only a rational map, we assume that vertices of $T$ are in general position in the sense specified in the course of the proof.

Consider a marked projective tetrahedron $T=(Q,\{H_1 \ldots, H_4\}).$ We construct the rational elliptic surface $(X_T, F_1, F_2)$ in two steps. First, we blow up the quadric $Q$  at the twelve points $E_{ij}$ and obtain a rational surface $R_T.$ We denote the class of the preimage of the exceptional divisor associated to the blow up at $E_{ij}$ by  $[E_{ij}]$ and the strict transforms of the generators of the quadric by $[L]$ and $[R].$ The surface $X_T$  is obtained from $R_T$ by  blowing down the following four $(-1)-$curves:
\begin{align*}
&[E_{21}],\\
&[L]-[E_{12}],\\
&[R]-[E_{12}],\\
&[L]+[R]-[E_{34}]-[E_{43}]-[E_{12}].
\end{align*}
We may assume that these lines are pairwise disjoint. Denote by $g\colon Q \dashrightarrow X_T$ the resulting rational map. 

\begin{lemma} \label{LemmaSurfaceD6}
The triple $\Bigl(X_T,g(H_3\cap H_4\cap Q),g(H_1\cap H_2\cap Q)\Bigr)$ is a $D_6$-surface.
\end{lemma}

\begin{proof}
The eight $(-1)$-curves 
\[
[E_{13}],\ [E_{31}], \ [E_{14}],\  [E_{41}],\ [E_{23}],\ [E_{32}], \ [E_{24}],\ [E_{42}],
\] 
are mutually disjoint on $X_T$, so they can be blown down simultaneously. By doing so, we obtain a del Pezzo surface with Picard number $1.$ The image of the strict transform of $H_1 \cap Q$  has self-intersection number equal to $2,$ so the surface is isomorphic to $\mathbb{P}^1 \times \mathbb{P}^1.$ The images of the curves $[E_{ij}]$ lie on a pair of reducible $(2,2)$-curves, which are the strict transforms of $(H_1 \cup H_2)\cap Q$ and $(H_3 \cup H_4)\cap Q,$ so $X_T$ is a rational elliptic surface with a pair of  $I_2-$fibers. 
\end{proof}

We introduce the following notation for the components of the fibers:
\begin{align*}
&F_{11}:=g(H_3\cap Q),\\
&F_{12}:=g(H_4\cap Q),\\
&F_{21}:=g(H_1\cap Q),\\
&F_{22}:=g(H_2\cap Q).
\end{align*}
and singular points
\begin{align*}
&F_1^x:=g([E_{21}]),\\
&F_1^y:=g([L]+[R]-[E_{12}]-[E_{34}]-[E_{43}]),\\
&F_2^x:=g([L]-[E_{12}]),\\
&F_2^y:=g([R]-[E_{12}]).
\end{align*}

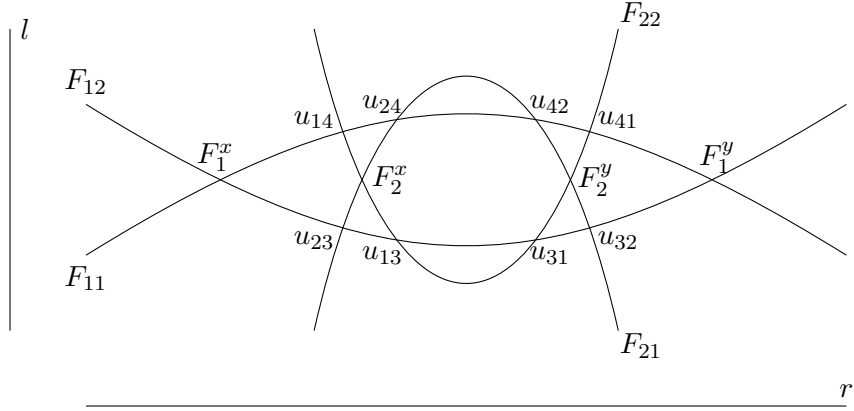
\begin{figure}
\centering
\begin{tikzpicture}
\draw (0,1) .. controls (4,3.5) and (6,3.5) .. (10,1);
\draw (0,3) .. controls (4,0.5) and (6,0.5) .. (10,3);
\draw (3,0) .. controls (4,4.5) and (6,4.5) .. (7,0);
\draw (3,4) .. controls (4,-0.5) and (6,-0.5) .. (7,4);
\draw (0,-1) -- (10,-1);
\draw (-1,0) -- (-1,4);
\draw (0,3.3) node{$F_{12}$};
\draw (0,0.7) node{$F_{11}$};
\draw (7.3,4.2) node{$F_{22}$};
\draw (7.3,-0.2) node{$F_{21}$};
\draw (1.7,2.3) node{$F_1^x$};
\draw (8.3,2.3) node{$F_1^y$};
\draw (4,2) node{$F_2^x$};
\draw (6.7,2) node{$F_2^y$};
\draw (-0.8,4) node{$l$};
\draw (10,-0.8) node{$r$};
\draw (3,2.8) node{$u_{14}$};
\draw (7,2.8) node{$u_{41}$};
\draw (3.9,3) node{$u_{24}$};
\draw (6.1,3) node{$u_{42}$};
\draw (3,1.2) node{$u_{23}$};
\draw (7,1.2) node{$u_{32}$};
\draw (3.9,1) node{$u_{13}$};
\draw (6.1,1) node{$u_{31}$};
\end{tikzpicture}
\caption{A pencil of $(2,2)$ curves on $\mathbb{P}^1\times \mathbb{P}^1$. Surface $X_T$ is obtained by blowing up points $u_{ij}.$}
\label{FigureEye}
\end{figure}

\begin{lemma}
The Picard lattice of $X_T$ is generated by the classes 
\begin{align*}
&l:=g([L]+[R]-[E_{12}]-[E_{34}]),\\
&r:=g([L]+[R]-[E_{12}]-[E_{43}]),\\
&u_{13}:=g([E_{13}]), \  u_{31}:=g([E_{31}]), \ u_{14}:=g([E_{14}]), \ u_{41}:=g([E_{41}]),\\
&u_{23}:=g([E_{23}]), \  u_{42}:=g([E_{42}]), \ u_{24}:=g([E_{24}]), \ u_{43}:=g([E_{43}]).
\end{align*}
The pairing is given on the classes by
\begin{align*}
&l^2=r^2=0,\  l.r=1, \ l.u_{ij}=r.u_{ij}=0,\\
& u_{ij}.u_{kl}=
\begin{cases}
&-1 \text{ if } u_{ij}=u_{kl}\\ 
&-1 \text{ if } u_{ij}\neq u_{kl}\\
\end{cases}.
\end{align*}
The canonical class is equal to
\[
K_X=-2l-2r+u_{13}+u_{31}+u_{14}+u_{41}+u_{23}+u_{32}+u_{24}+u_{42}.
\]
\end{lemma}

\begin{figure}
\centering
\begin{tikzpicture}

\draw (0,4) -- (12,4);
\draw (4,4) -- (4,2);

\draw (0,4) node[circle,fill=black,inner sep=0pt,minimum size=3pt]{};
\draw (0,4.4) node{\mbox{\scriptsize $-l-r+u_{24}+u_{42}+u_{14}+u_{41}$}};

\draw (2,4) node[circle,fill=black,inner sep=0pt,minimum size=3pt]{};
\draw (1.7,3.6) node{\mbox{\scriptsize $-l+u_{13}+u_{31}$}};

\draw (4,4) node[circle,fill=black,inner sep=0pt,minimum size=3pt]{};
\draw (4,4.4) node{\mbox{\scriptsize $l-u_{24}-u_{13}$}};

\draw (4,2) node[circle,fill=black,inner sep=0pt,minimum size=3pt]{};
\draw (4,1.6) node{\mbox{\scriptsize $l+r-u_{23}-u_{31}-u_{14}-u_{42}$}};

\draw (6,4) node[circle,fill=black,inner sep=0pt,minimum size=3pt]{};
\draw (6,3.6) node{\mbox{\scriptsize $l+r-u_{31}-u_{32}-u_{41}-u_{42}$}};

\draw (8,4) node[circle,fill=black,inner sep=0pt,minimum size=3pt]{};
\draw (8,4.4) node{\mbox{\scriptsize $r-u_{13}-u_{14}$}};

\draw (10,4) node[circle,fill=black,inner sep=0pt,minimum size=3pt]{};
\draw (10,3.6) node{\mbox{\scriptsize $l+r-u_{23}-u_{31}-u_{41}-u_{24}$}};

\draw (12,4) node[circle,fill=black,inner sep=0pt,minimum size=3pt]{};
\draw (12,4.4) node{\mbox{\scriptsize $l-u_{32}-u_{14}$}};

\draw (0,0) -- (12,0);
\draw (4,0) -- (4,-2);

\draw (0,0) node[circle,fill=black,inner sep=0pt,minimum size=3pt]{};
\draw (0,0.4) node{\mbox{\small $-e_{I}$}};

\draw (2,0) node[circle,fill=black,inner sep=0pt,minimum size=3pt]{};
\draw (2,-0.4) node{\mbox{\scriptsize $\dfrac{-e_{34}+e_{12}+e_{\varnothing}+e_I}{2}$}};

\draw (4,0) node[circle,fill=black,inner sep=0pt,minimum size=3pt]{};
\draw (4,0.4) node{\mbox{\scriptsize $\dfrac{e_{13}+e_{34}-e_{12}-e_{24}}{2}$}};

\draw (4,-2) node[circle,fill=black,inner sep=0pt,minimum size=3pt]{};
\draw (4,-2.4) node{\mbox{\scriptsize $\dfrac{e_{23}+e_{24}-e_{13}-e_{14}}{2}$}};

\draw (6,0) node[circle,fill=black,inner sep=0pt,minimum size=3pt]{};
\draw (6,-0.4) node{\mbox{\scriptsize $ \dfrac{e_{14}+e_{24}-e_{13}-e_{23}}{2}$}};

\draw (8,0) node[circle,fill=black,inner sep=0pt,minimum size=3pt]{};
\draw (8,0.4) node{\mbox{\scriptsize $\dfrac{e_{13}-e_{14}-e_{34}-e_{\varnothing}}{2}$}};

\draw (10,0) node[circle,fill=black,inner sep=0pt,minimum size=3pt]{};
\draw (10,-0.4) node{\mbox{\scriptsize $\dfrac{e_{14}+e_{23}-e_{13}-e_{24}}{2}$}};

\draw (12,0) node[circle,fill=black,inner sep=0pt,minimum size=3pt]{};
\draw (12,0.4) node{\mbox{\scriptsize $\dfrac{e_{12}+e_{34}-e_{14}-e_{23}}{2}$}};
\end{tikzpicture}

\caption{Matching root bases of lattices $f^{\perp}/f$ and $\textup{Q}(\textup{E}_8)$. We use the same notation for a class $r\in f^{\perp}$ and its projection $\pi(r)\in f^{\perp}/f.$ 
}
\label{FigureSimpleRoots2}
\end{figure}
\begin{proof}
This description comes from the presentation of $X_T$ as a blow up of $\mathbb{P}^1 \times \mathbb{P}^1$ in eight points described in the proof of Lemma \ref{LemmaSurfaceD6}, see Figure \ref{FigureEye}.
\end{proof}
To define the marking of $X_T$ it remains to construct an isometry between $\textup{Q}(\textup{E}_8)$ and  $f^{\perp}/f.$  We do that by identifying root bases, as in Figure \ref{FigureSimpleRoots2}. This finishes the construction of a rational map $Cor.$


\subsection{The correspondence commutes with length function}\label{SectionEqualityLengths}

\begin{proposition} \label{PropositionEqualityLengths} The map $M_{tetr} \dashrightarrow M_{surf}$ commutes with the projections to $\mathbb{T}.$ In other words, for a (general) marked projective tetrahedron $T$ and the corresponding surface $(X_T,F_1,F_2)$	 maps $\textup{L}_T$ and $\textup{Res}_{F_1}$ coincide.
\end{proposition}
\begin{proof}
It is easy to see that the lattice $\textup{Q}\bigl(\textup{E}_7^\textup{L}\bigr)$ is generated by the roots 
\begin{align*}
&\frac{\pm e_{23} \pm e_{24}\pm e_{34}\pm e_{\varnothing}}{2}, \ \ \frac{\pm e_{13}\pm e_{14} \pm e_{34}\pm e_{\varnothing}}{2},\ \ \frac{\pm e_{12}\pm e_{13} \pm e_{23}\pm e_{\varnothing}}{2}, \ \ 
\frac{\pm e_{12} \pm e_{14} \pm e_{24}\pm e_{\varnothing}}{2},
\end{align*}
so it is enough to check equality of $
\textup{L}_T$ and $\textup{Res}_{F_1}$
on them. The symmetries of the construction in \S \ref{SectionConstruction} allow us to reduce that even further and check the equality for only two roots
$
\dfrac{e_{23}+e_{24}+e_{34}+e_{\varnothing}}{2}$ and  $\dfrac{e_{12}+e_{13}+e_{23}+e_{\varnothing}}{2}.
$

\begin{lemma} We have
$
\textup{L}_T \left ( \dfrac{e_{23}+e_{24}+e_{34}+e_{\varnothing} }{2} \right )=\textup{Res}_{F_1}\left ( \dfrac{e_{23}+e_{24}+e_{34}+e_{\varnothing} }{2} \right).
$
\label{LemmaEquality1}
\end{lemma}

\begin{proof}
First, we have 
\[
 \dfrac{e_{23}+e_{24}+e_{34}+e_{\varnothing}}{2}=\pi(l-u_{23}-u_{42}).
\]
Consider the sections $[l-u_{23}]$ and $[u_{42}]$ of the elliptic fibration on  $X_T.$ 	They both intersect the component $F_{11}$ of the fiber $F_1.$ By (\ref{FormulaExplicitPeriod}) we have 
\begin{align*}
&\textup{Res}_{F_1}\left ( \dfrac{e_{23}+e_{34}+e_{24}+e_{\varnothing}}{2} \right)=
\textup{Res}_{F_1}([l-u_{23}]-[u_{42}])\\
=&\bigl[F_{1}^x,F_{11} \cap [l-u_{23}] ,F_{1}^y,F_{11} \cap [u_{42}]\bigr].
\end{align*}

Our goal is to compute the images of the points $F_{1}^x, F_{11} \cap [l-u_{23}], F_{1}^y,$ and $F_{11} \cap [u_{42}]$ under the birational isomorphism $g^{-1}$ restricted to $F_{11}.$
Denote by $U$ the intersection point of the plane $(A_3 A_4 E_{12})$ with $H_3 \cap Q$ which is different from $E_{12}.$ Similarly, denote by $V$ the intersection point of the plane $(E_{12}E_{34}E_{23})$  with $H_3 \cap Q$ which is different from $E_{12}.$ Under the birational isomorphism $g^{-1}$ the curve $F_{11}$ is mapped to the conic $H_3 \cap Q.$ 
Moreover, the restriction of $g^{-1}$ to $F_{11}$ sends $F_{1}^x$ to $E_{21},$   $F_{1}^y$ to $U,$  $F_{11} \cap [l-u_{23}]$ to $V$, and   $F_{11} \cap [u_{42}]$ to $E_{42}.$ Since $g^{-1}$ is a birational isomorphism, 

\[
\bigl[F_{1}^x,F_{11} \cap [l-u_{23}] ,F_{1}^y,F_{11} \cap [u_{42}]\bigr]=[E_{21}, V, U, E_{42}]_{H_3\cap Q}.
\]
Next, consider the projection from the point $E_{12}$ of  the conic $H_3\cap Q$ to the line $(A_2 A_4).$ We get that
\be
\begin{split}\label{FormulaEq1}
&\textup{Res}_{F_1}\left ( \dfrac{e_{23}+e_{24}+e_{34}+e_{\varnothing} }{2} \right)=[E_{21}, V, U, E_{42}]_{H_3\cap Q}\\
=&[A_2, (A_2 A_4)\cap (E_{23}E_{34}), A_4, E_{42})]_{(A_2 A_4)}.
\end{split}
\ee
On the other hand, Lemma \ref{LemmaProjectiveTriangle} implies that 
\be
\begin{split}\label{FormulaEq2}
&\textup{L}_T \left(\frac{ e_{23}+e_{24}+e_{34}+e_{\varnothing}}{2}\right)\\
=&\bigl[\widetilde{A}_2,E_{23},\widetilde{A}_3,E_{32}\bigr]\bigl[\widetilde{A}_2,E_{24},\widetilde{A}_4,E_{42}\bigr]\bigl[\widetilde{A}_3,E_{34},\widetilde{A}_4,E_{43}\bigr]\\
=&[A_4, E_{42},A_2, (A_2 A_4)\cap (E_{23}E_{34})]_{(A_2 A_4)}\\
=&[A_2, (A_2 A_4)\cap (E_{23}E_{34}), A_4, E_{42})]_{(A_2 A_4)}.
\end{split}
\ee
The statement of the lemma follows from (\ref{FormulaEq1}) and (\ref{FormulaEq2}).
\end{proof}

\begin{lemma}We have 
$
\textup{L}_T \left ( \dfrac{e_{12}+e_{13}+e_{23}+e_{\varnothing} }{2} \right )=\textup{Res}_{F_1}\left ( \dfrac{e_{12}+e_{13}+e_{23}+e_{\varnothing} }{2} \right).
$
\label{LemmaEquality2}
\end{lemma}

\begin{proof}
The proof is similar to the proof of the previous lemma. First, observe that
\[
 \frac{e_{12}+e_{13}+e_{23}+e_{\varnothing}}{2} =\pi(u_{31}-u_{23}).
\]
Next, consider the sections $u_{31}$ and $u_{23},$ which both intersect $F_{12}.$ By  (\ref{FormulaExplicitPeriod2}) we have
\begin{align*}
& \textup{Res}_{F_1}\left ( \dfrac{e_{12}+e_{13}+e_{23}+e_{\varnothing} }{2} \right)=
\textup{Res}_{F_1}([u_{31}]-[u_{23}])\\
=& [F_{1}^y,F_{12} \cap [u_{31}] ,F_{1}^x,F_{12} \cap [u_{23}]].
\end{align*}

First, we compute the images of  the points $F_{1}^y, F_{12} \cap [u_{31}], F_{1}^x$, and $F_{12} \cap [u_{23}]$ under the birational isomorphism $g^{-1}$ restricted to $F_{12}.$
Denote by $W$ the intersection point of the plane $(A_3 A_4 E_{12})$ with the conic $H_4 \cap Q$ which is different from $E_{12}.$ Under the birational isomorphism $g^{-1}$ the curve $F_{12}$ is mapped to the conic $H_4 \cap Q,$  $F_{1}^y$ to $W,$  $F_{12} \cap [u_{31}]$ to $E_{31},$  $F_{1}^x$ to $E_{21}$, and   $F_{12} \cap [u_{23}]$ to $E_{23}.$ Since $g^{-1}$ is a birational isomorphism, we have
\[
\bigl[F_{1}^y,F_{12} \cap [u_{31}] ,F_{1}^x,F_{12} \cap [u_{23}]\bigr]=[W, E_{31}, E_{21}, E_{23}]_{H_4\cap Q}.
\]

Next, consider the projection from the point $E_{12}$ of the conic $H_4\cap Q$ to the line $(A_2 A_3).$ We get that
\be \label{FormulaEq3}
\begin{split}
&\textup{Res}_{F_1}\left ( \dfrac{e_{12}+e_{13}+e_{23}+e_{\varnothing} }{2} \right)=[W, E_{31}, E_{21}, E_{23}]_{H_4\cap Q}\\
=&[A_3, (E_{12}E_{31})\cap (A_2A_3), A_2, E_{23}]_{(A_2 A_3)}.
\end{split}
\ee
 Lemma \ref{LemmaProjectiveTriangle} implies that 
\be \label{FormulaEq4}
\begin{split}
&\textup{L}_T \left(\frac{ e_{12}+e_{13}+e_{23}+e_{\varnothing}}{2}\right)\\
=&\bigl[\widetilde{A}_1,E_{12},\widetilde{A}_2,E_{21}\bigr]\bigl[\widetilde{A}_1,E_{13},\widetilde{A}_3,E_{31}\bigr]\bigl[\widetilde{A}_2,E_{23},\widetilde{A}_3,E_{32}\bigr]\\
=&[A_2, E_{23}, A_3, (E_{12}E_{31})\cap (A_2A_3) ]_{(A_2 A_3)}\\
=&[A_3, (E_{12}E_{31})\cap (A_2A_3), A_2, E_{23}]_{(A_2 A_3)}.
\end{split}
\ee
From (\ref{FormulaEq3}) and (\ref{FormulaEq4}) we get the result.
\end{proof}

Proposition \ref{PropositionEqualityLengths} follows from Lemmas \ref{LemmaEquality1} and \ref{LemmaEquality2}.
\end{proof}

\subsection{The correspondence commutes with angle function}\label{SectionEqualityAngles}

\begin{proposition} \label{PropositionEqualityAngles} For a (general) marked projective tetrahedron $T$ and the corresponding surface $(X_T,F_1,F_2)$ maps $\textup{A}_T$ and $\textup{Res}_{F_2}$ coincide.
\end{proposition}
\begin{proof}
Similarly to Proposition  \ref{PropositionEqualityLengths} in \S \ref{SectionEqualityLengths} it is sufficient to check the equality $
\textup{A}_T=\textup{Res}_{F_2}$ on the two roots
\[
\dfrac{e_{12}+e_{13}+e_{14}+e_I }{2}, \ \ \dfrac{e_{14}+e_{24}+e_{34}+e_I }{2}.
\]
We do that in Lemmas \ref{LemmaEquality3} and \ref{LemmaEquality4}.

\begin{lemma} We have
$\textup{A}_T \left ( \dfrac{e_{12}+e_{13}+e_{14}+e_I }{2} \right )=\textup{Res}_{F_2}\left ( \dfrac{e_{12}+e_{13}+e_{14}+e_I }{2} \right).$
\label{LemmaEquality3}
\end{lemma}
\begin{proof}
Observe  that
\[
\frac{e_{12}+e_{13}+e_{14}+e_I}{2}=\pi(u_{31}-u_{41}).
\]
Sections  $u_{31}$ and $u_{41}$ intersect the component $F_{22}$ of the fiber $F_2.$ We have 
\be \label{Equality45}
\begin{split}
& \textup{Res}_{F_2}\left ( \dfrac{e_{12}+e_{13}+e_{14}+e_I }{2} \right)=
\textup{Res}_{F_2}([u_{31}]-[u_{41}])	\\
=&\bigl[F_{2}^y,F_{22} \cap [u_{31}] ,F_2^x,F_{22} \cap [u_{41}]\bigr]_{F_{22}}.
\end{split}
\ee

Our goal is to compute the  images of the points $F_{2}^y, F_{22} \cap [u_{31}], F_2^x$ and $F_{22} \cap [u_{41}]$ under the birational isomorphism $g^{-1}$ restricted to $F_{22}.$
 The curve $F_{22}$ is mapped to the conic $H_2 \cap Q,$   $F_{2}^y$ to  $H_2 \cap R_{E_{12}},$  $F_{22} \cap [u_{31}]$ to $E_{31},$  $F_2^x$ to $L_{E_{12}}\cap H_2,$  and  $F_{22} \cap [u_{41}]$ to $E_{41}.$ Since $g^{-1}$ is a birational isomorphism, we have
\be \label{Equality5}
\begin{split}
\bigl[F_{2}^y,F_{22} \cap [u_{31}] ,F_2^x,F_{22} \cap [u_{41}]\bigr]_{F_{22}}=
[R_{E_{12}}\cap H_2, E_{31}, L_{E_{12}}\cap H_2, E_{41}]_{H_2\cap Q}.
\end{split}
\ee
Consider the map from $H_2\cap Q$ to $R_{E_{12}}$ sending the point $X \in H_2\cap Q$ to $L_X \cap R_{E_{12}}.$ This map  is an isomorphism, so
\be \label{Equality6}
\begin{split}
&[R_{E_{12}}\cap H_2, E_{31}, L_{E_{12}}\cap H_2, E_{41}]_{H_2\cap Q}	=[H_2 \cap R_{E_{12}},L_{E_{31}} \cap R_{E_{12}},E_{12} ,L_{E_{41}} \cap R_{E_{12}}]_{R_{E_{12}}}.
\end{split}
\ee
Next we apply the duality with respect to $Q.$ Since points of $Q$ are dual to the tangent planes at these points, we obtain the following equality of cross-ratios:
\be\label{Equality7}
\begin{split}
&[H_2 \cap R_{E_{12}},L_{E_{31}} \cap R_{E_{12}},E_{12} ,L_{E_{41}} \cap R_{E_{12}}]_{R_{E_{12}}}\\
=&[\langle H_2^{\vee}, R_{E_{12}}\rangle,\ \langle L_{E_{31}}, R_{E_{12}}\rangle,\ \langle L_{E_{12}}, R_{E_{12}}\rangle,\ \langle L_{E_{41}}, R_{E_{12}}\rangle ]_{R_{E_{12}}}.
\end{split}
\ee
Recall that for all $i, j$ we have $L_{E_{ij}}=L_{\widetilde{E}_{ij}}$ and $R_{E_{ij}}=R_{\widetilde{E}_{ji}},$  so
\be \label{Equality8}
\begin{split}
&[\langle H_2^{\vee}, R_{E_{12}}\rangle,\ \langle L_{E_{31}}, R_{E_{12}}\rangle,\ \langle L_{E_{12}}, R_{E_{12}}\rangle,\ \langle L_{E_{41}}, R_{E_{12}}\rangle ]_{R_{E_{12}}}\\
=&[\langle H_2^{\vee}, R_{\widetilde{E}_{21}}\rangle,\ \langle L_{\widetilde{E}_{31}}, R_{\widetilde{E}_{21}}\rangle,\ \langle L_{\widetilde{E}_{12}}, R_{\widetilde{E}_{21}}\rangle,\ \langle L_{\widetilde{E}_{41}}, R_{\widetilde{E}_{21}}\rangle ]_{R_{\widetilde{E}_{21}}}.
\end{split}
\ee
The last cross-ratio can be computed by intersecting the four planes with the line 
$(H^{\vee}_2H^{\vee}_4):$
\be \label{Equality9}
\begin{split}
&[\langle H_2^{\vee}, R_{\widetilde{E}_{21}}\rangle,\ \langle L_{\widetilde{E}_{31}}, R_{\widetilde{E}_{21}}\rangle,\ \langle L_{\widetilde{E}_{12}}, R_{\widetilde{E}_{21}}\rangle,\ \langle L_{\widetilde{E}_{41}}, R_{\widetilde{E}_{21}}\rangle ]_{R_{\widetilde{E}_{21}}}\\
=&[H_2^{\vee},\widetilde{E}_{31}, H_4^{\vee}, (\widetilde{E}_{41}\widetilde{E}_{21})\cap (H_2^{\vee}H_4^{\vee})]_{(H_2^{\vee}H_4^{\vee})}.
\end{split}
\ee
Combining (\ref{Equality45}), (\ref{Equality5}), (\ref{Equality6}), (\ref{Equality7}), (\ref{Equality8}), and  (\ref{Equality9}) we get 
\be \label{Equality10}
\textup{Res}_{F_2}\left ( \dfrac{e_{12}+e_{13}+e_{14}+e_I }{2} \right)\\
=[H_2^{\vee},\widetilde{E}_{31}, H_4^{\vee}, (\widetilde{E}_{41}\widetilde{E}_{21})\cap (H_2^{\vee}H_4^{\vee})]_{(H_2^{\vee}H_4^{\vee})}.
\ee

From Lemma \ref{LemmaProjectiveTriangle} we conclude that
\be \label{Equality11}
\begin{split}
&\textup{A}_{T}\left ( \dfrac{e_{12}+e_{13}+e_{14}+e_I }{2} \right)\\
=&\bigl[H_4^\vee,\widetilde{E}_{21},H_3^\vee,\widetilde{E}_{12}\bigr]
\bigl[H_3^\vee,\widetilde{E}_{41},H_2^\vee,\widetilde{E}_{14}\bigr]
\bigl[H_4^\vee,\widetilde{E}_{13},H_2^\vee,\widetilde{E}_{31}\bigr]\\
=&\bigl[H_2^{\vee},\widetilde{E}_{31}, H_4^{\vee}, (\widetilde{E}_{41}\widetilde{E}_{21})\cap (H_2^{\vee}H_4^{\vee})\bigr]_{(H_2^{\vee}H_4^{\vee})}.
\end{split}
\ee
The lemma follows from (\ref{Equality10}) and (\ref{Equality11}).
\end{proof}

\begin{lemma} We have $\textup{A}_T \left ( \dfrac{e_{14}+e_{24}+e_{34}+e_I }{2} \right )=\textup{Res}_{F_2}\left ( \dfrac{e_{14}+e_{24}+e_{34}+e_I }{2} \right).$
\label{LemmaEquality4}
\end{lemma}
\begin{proof}

It is easy to see that
\[
\frac{e_{14}+e_{24}+e_{34}+e_I}{2}=\pi(l-u_{42}-u_{41}).
\]
Consider the sections $[l-u_{41}]$ and $[u_{42}]$ intersecting $F_{21}$. We have
\be\label{Equality12}
\begin{split}
& \textup{Res}_{F_2}\left ( \dfrac{e_{14}+e_{24}+e_{34}+e_I }{2} \right)=
\textup{Res}_{F_2}([l-u_{41}]-[u_{42}])	\\
=&[F_2^x,F_{21} \cap [l-u_{41}] ,F_{2}^y,F_{21} \cap [u_{42}]]_{F_{21}}.
\end{split}
\ee

Our goal is to compute  the images of the  points $F_2^x, F_{21} \cap [l-u_{41}], F_{2}^y$ and $F_{21} \cap [u_{42}]$ under the birational isomorphism $g^{-1}$ restricted to $F_{21}.$ Denote by $U$ the point of intersection of the plane $(E_{12}E_{34}E_{41})$ with $H_1 \cap Q$, which is different from $E_{34}.$ Under the birational isomorphism $g^{-1}$ the curve $F_{21}$ is mapped to the conic $H_1 \cap Q,$   $F_2^x$ to  $L_{E_{12}}\cap H_1,$  $F_{21} \cap [l-u_{41}]$ to $U,$  $F_{2}^y$ to $H_1 \cap R_{E_{12}},$  and   $F_{21} \cap [u_{42}]$ to $E_{42}.$ Since $g^{-1}$ is a birational isomorphism we have 
\be\label{Equality13}
\begin{split}
\bigl[F_2^x,F_{21} \cap [l-u_{41}] ,F_{2}^y,F_{21} \cap [u_{42}]\bigr]_{F_{21}}
=[L_{E_{12}}\cap H_1, U, H_1 \cap R_{E_{12}}, E_{42}]_{H_1\cap Q}.
\end{split}
\ee
Next, consider the map from $H_1\cap Q$ to $R_{E_{12}}$ sending $X \in H_1\cap Q$ to $L_X \cap R_{E_{12}}.$ It is an isomorphism, so
\be\label{Equality14}
\begin{split}
&[L_{E_{12}}\cap H_1, U, H_1\cap R_{E_{12}}, E_{42}]_{H_1\cap Q}\\
=&[E_{12},L_U \cap R_{E_{12}}, H_1 \cap R_{E_{12}},L_{E_{42}} \cap R_{E_{12}}]_{R_{E_{12}}}.
\end{split}
\ee
We apply the duality with respect to $Q$:
\be\label{Equality15}
\begin{split}
&[E_{12},L_U \cap R_{E_{12}}, H_1 \cap R_{E_{12}},L_{E_{42}} \cap R_{E_{12}}]_{R_{E_{12}}}\\
=&[\langle L_{E_{12}}, R_{E_{12}}\rangle,\ \langle L_U, R_{E_{12}}\rangle,\ \langle H_1^{\vee}, R_{E_{12}}\rangle,\ \langle L_{E_{42}}, R_{E_{12}}\rangle ]_{R_{E_{12}}}.
\end{split}
\ee
We have $L_{E_{ij}}=L_{\widetilde{E}_{ij}}$ and $R_{E_{ij}}=R_{\widetilde{E}_{ji}},$  so
\be\label{Equality16}
\begin{split}
&[\langle L_{E_{12}}, R_{E_{12}}\rangle,\ \langle L_U, R_{E_{12}}\rangle,\ \langle H_1^{\vee}, R_{E_{12}}\rangle,\ \langle L_{E_{42}}, R_{E_{12}}\rangle ]_{R_{E_{12}}}\\
=&[\langle L_{\widetilde{E}_{12}}, R_{\widetilde{E}_{21}} \rangle,\ \langle L_{U}, R_{\widetilde{E}_{21}}\rangle,\ \langle H_1^{\vee}, R_{\widetilde{E}_{21}}\rangle,\ \langle L_{\widetilde{E}_{42}}, R_{\widetilde{E}_{21}}\rangle ]_{R_{\widetilde{E}_{21}}}.
\end{split}
\ee
The last cross-ratio can be computed by intersecting the four planes with the line 
$(H^{\vee}_1H^{\vee}_3):$
\be\label{Equality17}
\begin{split}
&[\langle L_{\widetilde{E}_{12}}, R_{\widetilde{E}_{21}} \rangle,\ \langle L_{U}, R_{\widetilde{E}_{21}}\rangle,\ \langle H_1^{\vee}, R_{\widetilde{E}_{21}}\rangle,\ \langle L_{\widetilde{E}_{42}}, R_{\widetilde{E}_{21}}\rangle ]_{R_{\widetilde{E}_{21}}}\\
=&[H_3^{\vee}, \langle U, R_{E_{12}}\rangle \cap (H_1^{\vee}H_3^{\vee}), H_1^{\vee}, \widetilde{E}_{42}]_{(H_1^{\vee}H_3^{\vee})}.
\end{split}
\ee

Combining (\ref{Equality12}), (\ref{Equality13}), (\ref{Equality14}), (\ref{Equality15}), (\ref{Equality16}), and  (\ref{Equality17}) we get 
\be \label{Equality18}
\textup{Res}_{F_2}\left ( \dfrac{e_{14}+e_{24}+e_{34}+e_I }{2} \right)\\
=[H_3^{\vee}, \langle U, R_{E_{12}}\rangle \cap (H_1^{\vee}H_3^{\vee}), H_1^{\vee}, \widetilde{E}_{42}]_{(H_1^{\vee}H_3^{\vee})}.
\ee

 On the other hand, we have
\begin{align*}
&\textup{A}_{T}\left ( \dfrac{e_{14}+e_{24}+e_{34}+e_I }{2} \right)\\
=&\bigl[H_3^\vee,\widetilde{E}_{41},H_2^\vee,\widetilde{E}_{14}\bigr]
\bigl[H_2^\vee,\widetilde{E}_{43},H_1^\vee,\widetilde{E}_{34}\bigr]
\bigl[H_3^\vee,\widetilde{E}_{24},H_1^\vee,\widetilde{E}_{42}\bigr]\\
=&[H_3^{\vee}, (\widetilde{E}_{41}\widetilde{E}_{43})\cap (H_1^{\vee}H_3^{\vee}) ,H_1^{\vee}, \widetilde{E}_{42} ]_{(A_2 A_3)}.
\end{align*}

To prove the lemma we need to show that the lines $(\widetilde{E}_{41}\widetilde{E}_{43}),\ (H_1^{\vee}H_3^{\vee})$ and the plane $\langle U, R_{E_{12}}\rangle$ intersect in one point. Consider point $V=(E_{12}E_{41})\cap (UE_{34}).$ The plane $V^{\vee}$ passes through the line $(H_{1}^{\vee}H_3^{\vee}).$
Since $V$ is contained in the plane $\langle R_{E_{12}},L_{E_{41}}\rangle,$ the plane $V^{\vee}$ contains the point $W_1=R_{\widetilde{E}_{21}} \cap L_{\widetilde{E}_{41}}.$ Similarly, $V$ is contained in the plane $\langle L_U, R_{E_{34}}\rangle,$ so $V^{\vee}$ contains $W_2=L_U \cap R_{\widetilde{E}_{43}}.$ 

Consider the point $W=(W_1W_2)\cap (H_1^{\vee}H_3^{\vee}).$  We claim that 
\be \label{LastEquality}
W=(\widetilde{E}_{41}\widetilde{E}_{43}) \cap (H_1^{\vee}H_3^{\vee}) \cap \langle U, R_{E_{12}}\rangle.
\ee
Indeed, clearly $W\in (H_1^{\vee}H_3^{\vee}).$ Since
\[
(W_1W_2)=\langle L_U ,R_{\widetilde{E}_{21}}\rangle \cap \langle R_{\widetilde{E}_{43}} ,L_{\widetilde{E}_{41}}\rangle,
\]
the point $W$ is contained in the  plane $\langle L_U ,R_{\widetilde{E}_{21}}\rangle=\langle U ,R_{E_{12}}\rangle.$ By the same reason, $W$ is contained in  the plane $\langle R_{\widetilde{E}_{43}} ,L_{\widetilde{E}_{41}}\rangle.$ Since $W\in (H_1^\vee H_2^\vee H_3^\vee),$ we have 
\[
W \in \langle R_{\widetilde{E}_{43}} ,L_{\widetilde{E}_{41}}\rangle \cap (H_1^\vee H_2^\vee H_3^\vee)=(\widetilde{E}_{41}\widetilde{E}_{43}),
\] 
which implies $(\ref{LastEquality}).$ From here the lemma follows. 
\end{proof}
Proposition \ref{PropositionEqualityAngles} follows from Lemmas \ref{LemmaEquality3} and \ref{LemmaEquality4} .
\end{proof}

\bibliographystyle{alpha}      
\bibliography{Tetrahedra_Bibliography}

\end{document}